\documentclass[12pt]{amsart}
\usepackage{tikz-cd}
\usepackage{amsmath}
\usepackage{fourier}
\usepackage{amssymb}
\usepackage{amscd}
\usepackage{amsthm}
\usepackage[centertags]{amsmath}
\usepackage{amsfonts}
\usepackage{newlfont}
\usepackage{graphicx}
\usepackage{amsfonts, amssymb}
\usepackage{mathrsfs}
\usepackage{latexsym}
\usepackage{tikz}
\usepackage{verbatim}
\usepackage[all]{xy}
\usepackage{enumitem}
\usepackage[colorlinks=true,linkcolor=colorref,citecolor=colorcita,urlcolor=colorweb]{hyperref}
\definecolor{colorcita}{RGB}{21,86,130}
\definecolor{colorref}{RGB}{5,10,177}
\definecolor{colorweb}{RGB}{177,6,38}

\numberwithin{subsection}{section}

\newtheorem{theorem}{Theorem}[section]

\newtheorem{proposition}[theorem]{Proposition}
\newtheorem{corollary}[theorem]{Corollary}
\newtheorem{lemma}[theorem]{Lemma}

\theoremstyle{definition}
\newtheorem{remark}[theorem]{Remark}

\theoremstyle{remark}

\usepackage{mathtools}
\DeclarePairedDelimiter{\abs}{\lvert}{\rvert}

\DeclareMathOperator{\spa}{span}
\DeclareMathOperator{\id}{\mathrm{id}}
\DeclarePairedDelimiter\floor{\lfloor}{\rfloor}

\newcommand\restrict[1]{\raisebox{-.5ex}{$\vert$}_{#1}}

\begin{document}
\title[Minimal projections onto spaces of polynomials on  real euclidean   spheres]{Minimal projections onto spaces of polynomials\\ on  real euclidean   spheres}

\author[Defant]{A.~Defant}
\address{%
Institut f\"{u}r Mathematik,
Carl von Ossietzky Universit\"at,
26111 Oldenburg,
Germany}
\email{defant$@$mathematik.uni-oldenburg.de}

\author[Galicer]{D.~Galicer}
\address{Departamento de Matem\'{a}ticas y Estad\'{\i}stica, Universidad Torcuato Di Tella, Av. Figueroa Alcorta 7350 (1428), Buenos Aires, Argentina and IMAS-CONICET. \tiny{On leave from Departamento de Matematica, Facultad de Ciencias Exactas y Naturales, Universidad de Buenos Aires, (1428) Buenos Aires, Argentina}}  \email{daniel.galicer@utdt.edu}

\author[Mansilla]{M.~Mansilla}
\address{Departamento de Matem\'{a}tica,
Facultad de Cs. Exactas y Naturales, Universidad de Buenos Aires and IAM-CONICET. Saavedra 15 (C1083ACA) C.A.B.A., Argentina}
\email{mmansilla$@$dm.uba.ar}

\author[Masty{\l}o]{M.~Masty{\l}o}
\address{Faculty of Mathematics and Computer Science, Adam Mickiewicz University, Pozna{\'n}, Uniwersytetu
\linebreak
Pozna{\'n}skiego 4,
61-614 Pozna{\'n}, Poland}
\email{mieczyslaw.mastylo$@$amu.edu.pl}

\author[Muro]{S.~Muro}
\address{FCEIA, Universidad Nacional de Rosario and CIFASIS, CONICET, Ocampo $\&$ Esmeralda, S2000 Rosario, Argentina}
\email{smuro$@$fceia.unr.edu.ar}

\date{}

\thanks{The research of the fourth author was supported by the National Science Centre (NCN), Poland, Project 2019/33/B/ST1/00165. The research of the second and fifth author is additionally supported by 20020220300242BA}

\begin{abstract}
We investigate projection constants within classes of multivariate polynomials over finite-dimensional real Hilbert spaces. Specifically, we consider the projection constant for spaces of spherical harmonics  and spaces of homogeneous polynomials as well as for spaces of polynomials of finite degree on the unit sphere. We  establish a connection between these quantities and certain weighted $L_1$-norms of specific Jacobi polynomials. As a consequence, we present exact formulas, computable expressions and  asymptotically accurate estimates  for them.  
\end{abstract}


\maketitle

\maketitle

\section{Introduction}

Spherical harmonics and spherical polynomials have become indispensable tools across scientific disciplines, providing elegant solutions for problems defined on the unit sphere and beyond. From solving partial differential equations to enhancing numerical integration methods, their versatility meets the demands of diverse applications.

Exploring approximation techniques and numerical methods over spheres reflects the growing interest in efficient computational algorithms. 
From a functional analysis perspective, particularly within the realm of Banach space theory, it is crucial to investigate the projection constant for Banach spaces of spherical harmonics and spherical polynomials.

The concept of the projection constant is foundational in the field of Banach spaces and their local theory. Its origins trace back to the examination of complemented subspaces within Banach spaces.

Consider $X$ as a complemented subspace of a Banach space $Y$. The relative projection constant of $X$ in $Y$ is defined as follows:
\begin{align*}
\boldsymbol{\lambda}(X, Y) & =  \inf\big\{\|P\|: \,\, P\in \mathcal{L}(Y, X),\,\, P|_{X} = \id_X\big\}\\
&  = \inf\big\{c>0: \,\,\text{$\forall\, T \in \mathcal{L}(X, Z)$ \,\, $\exists$\,\, an extension\,
$\widetilde{T}\in \mathcal{L}(Y, Z)$\, with $\|\widetilde{T}\| \leq c\,\|T\|$}\big\}\,,
\end{align*}
where $\id_X$ denotes the identity operator on $X$, and $\mathcal{L}(U,V)$ represents the Banach space of all bounded linear operators between Banach spaces $U$ and $V$ with the operator norm. Here, $\mathcal{L}(U):=\mathcal{L}(U,U)$, and we adopt the convention that $\inf \varnothing = \infty$. This equality between the projection constant and the so-called extension constant can be found, for example, in \cite[III.B.4 Lemma]{wojtaszczyk1996banach}.

The (absolute) projection constant of $X$ is expressed as:
\[
\boldsymbol{\lambda}(X) := \sup \,\,\boldsymbol{\lambda}(I(X),Y)\,,
\]
where the supremum is taken over all Banach spaces $Y$ and isometric embeddings $I\colon X \to Y$. If $X$ is a finite-dimensional Banach space and $X_1$ is a subspace of some $C(K)$-space isometric to $X$, then (see, e.g. \cite[III.B.5 Theorem] {wojtaszczyk1996banach}):
\begin{align}\label{eq: l_infty inyective}
\boldsymbol{\lambda}(X) = \boldsymbol{\lambda}(X_1, C(K))\,.
\end{align}
Thus, determining $\boldsymbol{\lambda}(X)$ is equivalent to finding the norm of a~minimal projection from $C(K)$ onto~$X_1$.

It is worth noting that projections play an important role in approximation theory. In fact, if $Y$ is a~Banach space and $P$ a~projection from $Y$ onto a subspace $X$, then the error $\|y - Py\|_Y$ of approximation of an element $y\in Y$ by $Py$ satisfies
\[
\|y- Py\|_Y \leq \|\id_Y - P\| \, \text{dist}(y, X) \leq (1 + \|P\|)\, \text{dist}(y, X)\,,
\]
where $\text{dist}(y, X) =\inf\{\|y-x\|_Y \,: \, x\in X\}$. This estimate motivates the problem of minimizing $\|P\|$, and any
projection  $P_0\colon Y \to Y$ onto $X$ such that $\|P_0\| = \boldsymbol{\lambda}(X, Y)$, is said to be a minimal projection of $Y$ onto $X$.

In $L^2(\mathbb{S}^{n-1})$, the space of square-integrable functions on the unit sphere $\mathbb{S}^{n-1}$ in the Euclidean space $\mathbb{R}^n$, each function can be decomposed into an orthogonal sum of homogeneous spherical harmonics, serving as fundamental building blocks. Understanding the projection constant in this context aids in evaluating the efficacy of a partial approximation for a given function.

The initial inspiration for this work stems from the seminal contributions of Ryll and Wojtaszczyk \cite{ryll1983homogeneous}. Solving a problem posed by Waigner, they demonstrated that the inclusion $H_\infty(B_{\ell_2^n(\mathbb{C})}) \hookrightarrow H_1(B_{\ell_2^n(\mathbb{C})})$ between Hardy spaces is not compact, where $B_{\ell_2^n(\mathbb{C})}$ denotes the open unit ball of the complex $n$-dimensional Hilbert space $\ell_2^n(\mathbb{C})$ and $n>1$. As noted in \cite{ryll1983homogeneous}, this result is intimately related to a question raised by Rudin in his monograph \cite{rudin1980}: Does there exist an inner function on the open unit ball of the Hilbert space $\ell_2^n(\mathbb{C})$, $n>1$ ?  Recall that a nonconstant bounded holomorphic function with domain $B_{\ell_2^n(\mathbb{C})}$ is called an inner function if its radial limits have modulus equal to $1$ almost everywhere on $\mathbb{S}^{n-1}$, where "almost everywhere" refers to the rotation-invariant probability measure on $\mathbb{S}^{n-1}$. For more details, we refer the reader to \cite{rudin1985} and \cite{wojtaszczyk1996banach}.

Much of this deep cycle of ideas  relies on the following concrete formula for the projection
constant  of the Banach space $\mathcal{P}_{d}(\ell_2^n(\mathbb{C}))$ of all  $d$-homogeneous polynomials on $\ell_2^n(\mathbb{C})$:
\begin{align} \label{fascinating}
\boldsymbol{\lambda}\big(\mathcal{P}_d(\ell_2^n(\mathbb{C})\big) = \frac{\Gamma(n +d) \Gamma(1 + \frac{d}{2})}{\Gamma(1 + d)\Gamma(n + \frac{d}{2})}, \quad\, d,n\in \mathbb{N}\,
\end{align}
(see \cite{ryll1983homogeneous}, and also  \cite[III.B.15]{wojtaszczyk1996banach}
and \cite{rudin1985})). The norm considered in $\mathcal{P}_{d}(\ell_2^n(\mathbb{C}))$ is, as usual, the supremum norm taken with respect to  the complex euclidean ball $B_{\ell_2^n(\mathbb{C})}$.

A~simple calculation yields, 
\begin{align} \label{surprise}
\boldsymbol{\lambda}\big(\mathcal{P}_d(\ell_2^n(\mathbb{C})\big) \leq 2^{n-1}, \quad\, d,n\in \mathbb{N}\,.
\end{align}
The case $n=2$ is of particular interest. Indeed, in the mentioned paper \cite{ryll1983homogeneous} the authors noticed
the~surprising fact that the sequence $\big(X_d\big)_{d\geq 1}$ with $X_d:= \mathcal{P}_{d}(\ell_2^2(\mathbb{C}))$ forms the first nontrivial known example of a sequence of finite-dimensional Banach spaces for which $\lim_{d \to \infty} \dim X_d =\infty$ although 
\[
\sup_d \boldsymbol{\lambda}(X_d) <~\infty\,. 
\]
It seems interesting to note that the formula~\eqref{fascinating}
combined with the well-known limit shown further in Equation~\eqref{mainasym} gives that 
\begin{align} \label{surpriseB}
 \lim_{d\to \infty} \boldsymbol{\lambda}\big(\mathcal{P}_{d}(\ell_2^n(\mathbb{C})\big) =2^{n-1}\,.
\end{align}
Thus, more precisely, we have $\sup_d \boldsymbol{\lambda}(X_d)=2$. It is worth noting here that Bourgain \cite{bourgain1989} gave an affirmative solution to a~problem considered in \cite{ryll1983homogeneous},
showing that the sequence $\big(\text{dist}\big(X_d, \ell_\infty^{\dim X_d}(\mathbb{C})\big)\big)_d$ of Banach Mazur distances is bounded.

What about real Hilbert spaces $\ell_2^n(\mathbb{R})$? By Rutovitz \cite{rutovitz}  we know that 
\begin{align} \label{beginning}
\boldsymbol{\lambda}\big(\ell_2^n(\mathbb{R})\big)  
= \frac{2}{\sqrt{\pi}}   \frac{\Gamma(\frac{n+2}{2})}{\Gamma(\frac{n+1}{2})}\,,
\end{align}
which is different from the complex case $d=1$ in 
\eqref{fascinating}:
\begin{align} \label{Ruto}
\boldsymbol{\lambda}\big(\ell_2^n(\mathbb{C})\big)
= \frac{\sqrt{\pi}}{2}   \frac{n!}{\Gamma(n + \frac{1}{2})}\,.
\end{align}
At first glance, one might think that obtaining the projection constant of $ \ell_2^n(\mathbb{C}) $ is as simple as replacing $ n $ with $ 2n $ in \eqref{beginning}, given that $ \ell_2^n(\mathbb{C}) $ is isometrically isomorphic to $ \ell_2^{2n}(\mathbb{R}) $. In fact,  replacing $n$ with $2n$ does not yield the correct value. 
Note that the projections considered in the definitions are $\mathbb{R}$-linear in one setting and $\mathbb{C}$-linear in the other.

Among others, motivated by the fascinating formula from \eqref{fascinating}, our main goal here is to study the projection constant of $\mathcal{P}_{d}(\ell_2^n(\mathbb{R}))$.

Contrary to initial expectations, the real case unfolds with intricate subtleties, ultimately revealing nuanced differences from its complex counterpart.
To illustrate one of these differences, we note that by~\eqref{surprise} the sequence
  $\big(\boldsymbol{\lambda}\big(\mathcal{P}_{d}(\ell_2^n(\mathbb{C}))\big)\big)_d$  is bounded for each fixed $n$,
whereas we are going to show that in the real case  for each fixed $n\geq 2$ the sequence $\big(\boldsymbol{\lambda}\big(\mathcal{P}_{d}(\ell_2^n(\mathbb{R}))\big)\big)_d$ increases  to infinity.

Our focus in this article lies on finding
 suitable extensions of \eqref{fascinating} for various spaces of multivariate polynomials on finite-dimensional real Hilbert spaces  - more precisely, we intend  to present  concrete
formulas or computable expressions for the projection constants of various classes of polynomials defined on finite-dimensional real Hilbert spaces or their  spheres.
To this end, we show that the projection constant can be expressed through an average integral representation. Building on this observation, we prove that the optimal projection minimizing the norm coincides with the orthogonal projection in $L_2(\mathbb{S}^{n-1})$. This perspective allows us to obtain integral representations for the projection constants under study and use this to derive asymptotically sharp estimates.

In particular, we show that the projection constants of the spaces $\mathcal{P}_{\leq d}(\mathbb{S}^{n-1})$, $\mathcal{P}_{ d}(\mathbb{S}^{n-1})$, and $\mathcal{H}_{d}(\mathbb{S}^{n-1})$ for $n>2$ behave asymptotically as $d^{\frac{n-2}{2}}$ for large $d$, despite their considerable dimensional differences. These spaces represent the Banach spaces of degree-$d$ polynomials, $d$-homogeneous polynomials, and $d$-homogeneous spherical harmonics, respectively, endowed with the supremum norm on the real unit sphere $\mathbb{S}^{n-1}$
of $\ell_2^n(\mathbb{R})$. 
Consequently, we observe that for the first two of these spaces the Kadets-Snobar upper bound is far from being tight. This estimate asserts that the projection constant of any finite-dimensional space is bounded by the square root of its dimension (see, for instance, \cite[Theorem 10, III.B.]{wojtaszczyk1996banach}).
We also analyze the behavior of the ratio of their projection constant and $d^{\frac{n-2}{2}}$ as a function of the dimension $n$ when $d$ is sufficiently large.
The case $n=2$ is also addressed. Specifically, we prove that the projection constant of $\mathcal{P}_{\leq d}(\mathbb{S}^{1})$ and $\mathcal{P}_{d}(\mathbb{S}^{1})$ behaves as $\log d$ as $d\to \infty$, whereas the one for $\mathcal{H}_{d}(\mathbb{S}^{1})$ remains constant.

The present work combines ideas from Banach space theory and approximation theory on the unit sphere---two areas that, while overlapping in some respects, differ significantly in terminology, focus, and methods. To ensure clarity in presenting our results, we provide detailed exposition of key tools and concepts, some of which may be familiar to specialists in one field but not the other. This unified presentation reflects how techniques from both areas come together in our setting and sets the stage for the main results that follow. To conclude this introduction, we provide a selection of recent literature relevant to the study of the projection constant: \cite{basso2019computation, blom2025properties, defant2023asymptotic, defant2023integral, defant2024projection,  deregowska2022value, derȩgowska2023simple, foucart2017maximal, foucart2018determining, HR2022, hokamp2023spaces, krieg2026sampling,  lewickimastylo, lewicki2021chalmers, lewicki2022codimension,  naor2021extension}. This compilation is not intended to be exhaustive but rather to serve as a starting point for exploring recent developments and the ongoing interest in this area.

We emphasize from the outset that the spaces of polynomials considered in this work do not include the classical spaces of algebraic polynomials on $[-1,1]$, whose projection constants do not appear to admit an analytic expression, as evidenced by the pivotal work of \cite{chalmers1990determination}.

\section{Preliminaries} We recommend consulting Wojtaszczyk's book  \cite{wojtaszczyk1996banach} for matters related to Banach space theory and Atkinson and Han's monograph \cite{atkinson2012spherical} for topics concerning spherical harmonics and approximation theory.

We denote the Lebesgue measure on $\mathbb{R}^n : = \ell_2^n(\mathbb{R})$ by $\lambda_n$, and the surface (Borel) measure on the real unit sphere $\mathbb{S}^{n-1}$ in $\ell_2^n(\mathbb{R})$ by $s_n$. Recall one of the simplest definitions of $s_n$:  If $A$ is a Borel set in $\mathbb{S}^{n-1}$, then 
\[
s_n(A) := n\,\lambda_n(\{rx: \, r\in (0, 1),\, \, x\in A\})\,.
\]  
Since Lebesgue measure of the euclidean unit ball is $\frac{\pi^{n/2}}{\Gamma( 1 + n/2 )}$ (see e.g. \cite[Chapter 1, eq.(1.17)]{pisier1999volume}), it follows that 
\begin{equation} \label{omega}
\pmb{\omega}_n :=s_n(\mathbb{S}^{n-1}) = \frac{2 \pi^{n/2}}{\Gamma(n/2)}\,.
\end{equation}
Throughout the paper, the measure $\frac{1}{\pmb{\omega}_n} s_n$ is called the normalized surface measure 
on $\mathbb{S}^{n-1}$ and is denoted, as usual, by $\sigma_n$.

The orthogonal group acting on  $\mathbb{S}^{n-1}$ is denoted by~$\mathcal{O}_n$. For a compact Hausdorff space $K$, we denote by $C(K)$ the Banach space of all complex-valued continuous functions defined on $K$ endowed with the supremum norm. 

Recall that the normalized surface measure $\sigma_n$ is invariant under orthogonal transformations. Thus, it can also be obtained by the useful formula 
\begin{equation}\label{greatrudy}
\int_{\mathbb{S}^{n-1}} f(\xi)\,d\sigma_n(\xi) = \int_{\mathcal{O}_n} f(A\xi_0)\,d\textbf{m}(A), \quad\, f\in C(\mathbb{S}^{n-1})\,,
\end{equation}
where $\textbf{m}$ is the normalized Haar measure on $\mathcal{O}_n$ and $\xi_0$  some fixed vector in $\mathbb{S}^{n-1}$.

We will use throughout the following notation: given two double sequences of positive numbers $(a_{n,d})_{n,d \in \mathbb{N}}$ and $(b_{n,d})_{n,d \in \mathbb{N}}$, we write
\[ 
a_{n,d} \sim_{c(n)} b_{n,d}\,, 
\]
if there exist constants $c_1(n), c_2(n) > 0$ that depend solely on $n$ (and not on $d$) such that
\[ c_1(n) a_{n,d} \leq b_{n,d} \leq c_2(n) a_{n,d} \]
for all $n, d \in \mathbb{N}$. The notation $a_{n,d} \sim_{c(d)} b_{n,d}$ is defined analogously.

In what follows, for a given sequence $(X_k)_{k=1}^{\infty}$ of linear subspaces of a linear space $X$, we define, for brevity, the notation
\[
\text{span}_k X_k := \text{span} \, \bigg( \bigcup_{k=1}^\infty X_k \bigg)\,.
\]
In the case when $X$ is a Banach space, the closure of $\text{span}_k X_k$ is denoted by $\overline{\text{span}}_{k} X_k$.

\subsection{Polynomials and spherical harmonics}
We write  $\mathcal P (\mathbb{R}^n)$ for the linear space of all  polynomials $f\colon \mathbb{R}^n \to \mathbb{C}$ of the form
 \begin{align} \label{repr}
f\left(x\right)=\sum_{\alpha \in J} c_\alpha\,x^\alpha, \quad\, x\in \mathbb{R}^n\,,
\end{align}
where  $J \subset  \mathbb{N}_0^n$ is finite, and $(c_\alpha)_{\alpha\in J}$ are  complex coefficients.  
Since $c_\alpha=\partial^\alpha f(0)/\alpha!$ for $\alpha \in J$, it is immediate that the preceding representation  is unique.
For $f \in \mathcal P (\mathbb{R}^n)$ we call 
\[
\text{deg}(f)
:= \max_{\alpha \in J} |\alpha|
\]
the degree of $f$,
where as usual $|\alpha|:=\sum \alpha_i$
for $\alpha \in \mathbb{N}^n_0$.

Given  $d \in \mathbb{N}_0$, we write
$\mathcal P_d (\mathbb{R}^n)$  for all $d$-homogeneous polynomials in $\mathcal P (\mathbb{R}^n)$
(that is, $f$ has a~representation like in \eqref{repr} with coefficients $ c_\alpha \neq 0$ only if $|\alpha|= d $),
and $\mathcal P_{\leq d}(\mathbb{R}^n)$ for all polynomials of 
degree at most $d$. Due to the linear independence of monomials $x^\alpha$, any degree-$d$ polynomial $f \in \mathcal P_{\leq d} (\mathbb{R}^n)$ may be uniquely represented as a~sum $f= \sum_{k=0}^d f_k$ of $k$-homogeneous polynomials $f_k$.
 Clearly,
\begin{equation} \label{klm1}
 \mathcal P_{\leq d}(\mathbb{R}^n)
=
  \spa_{k \leq d} \mathcal P_{k}(\mathbb{R}^n)
\quad \text{and} \quad
 \mathcal P(\mathbb{R}^n)
=
  \spa_k \mathcal P_k(\mathbb{R}^n)\,.
\end{equation}
Recall that
\begin{equation}\label{dimformuA}
    \dim \mathcal{P}_d(\mathbb{R}^n)  
    =  \binom{n+d-1}{n-1}=  \binom{n+d-1}{d}\,. 
  \end{equation}

The  linear spaces $\mathcal{P}_{d} (\mathbb{R}^n)$ and $\mathcal{P}_{\leq d}(\mathbb{R}^n)$ equipped  with the supremum norm over the real unit euclidean ball $B_{\ell_2^n(\mathbb{R})}$ form complex Banach spaces, which we  denote  by  
$\mathcal{P}_{d}(\ell_2^n(\mathbb{R}))$
and $\mathcal{P}_{\leq d}(\ell_2^n(\mathbb{R}))$ respectively. 

We denote by $\mathcal{P}(\mathbb{S}^{n-1})$ the linear space consisting of all restrictions $f|_{\mathbb{S}^{n-1}}$ of polynomials $f \in \mathcal{P}(\mathbb{R}^n)$. When only homogeneous polynomials of degree $d$ (or polynomials of degree at most $d$) are considered, we write $\mathcal{P}_d(\mathbb{S}^{n-1})$ or $\mathcal{P}_{\le d}(\mathbb{S}^{n-1})$, respectively.  Together with the supremum norm taken on $\mathbb{S}^{n-1}$,
both spaces $\mathcal{P}_d(\mathbb{S}^{n-1})$ or $\mathcal{P}_{\leq d}(\mathbb{S}^{n-1})$
  form  finite-dimensional subspaces of $C(\mathbb{S}^{n-1})$. 
Clearly,  by~\eqref{klm1}
\begin{equation}\label{ludoABAB}
\mathcal P_{\leq d} (\mathbb{S}^{n-1})
=
  \spa_{k \leq d} \mathcal P_{k}(\mathbb{S}^{n-1})
\quad \,\text{and} \,\quad
    \mathcal P ( \mathbb{S}^{n-1})=\spa_{k}\mathcal{P}_k (\mathbb{S}^{n-1})\,.
\end{equation}

Note that  the restriction map
\begin{equation}\label{aprilo}
\mathcal{P}_{d}(\ell_2^n(\mathbb{R}))  \to \mathcal{P}_{d}(\mathbb{S}^{n-1})\,, \quad\,
f \mapsto f_{|\mathbb{S}^{n-1}}
\end{equation}
is an isometric bijection; indeed, by homogeneity for each $f \in \mathcal{P}_{d}(\ell_2^n(\mathbb{R}))$
\[
\sup_{x\in B_{\ell_2^n(\mathbb{R})}} |f(x)| = \sup_{x \in \mathbb{S}^{n-1}} |f(x)|\,.
\]
On the other hand,
since $\sum_{k=1}^n x_k^2=1$ on $\mathbb{S}^{n-1}$, the surjective restriction map from $\mathcal{P}_{\leq d}(\ell_2^n(\mathbb{R}))$
to $\mathcal{P}_{\leq d}(\mathbb{S}^{n-1})$
 is  non-injective for $d \geq 2$, implying 
that $ \mathcal{P}_{\leq d}(\ell_2^n(\mathbb{R}))
\,\neq\,\mathcal{P}_{\leq d}(\mathbb{S}^{n-1})
\,.$

A polynomial $f \in  \mathcal{P}(\mathbb{R}^n)$ is said to be harmonic, whenever $\Delta f=0$, where as usual
\[
\Delta : = \sum_{j=1}^n \frac{\partial^2}{\partial x_{j}^2} \colon
\,\, \mathcal P (\mathbb{R}^n) \to \mathcal P (\mathbb{R}^n)
\]
denotes  the Laplace operator. We  write $\mathcal{H}(\mathbb{R}^n)$ for the subspace of all harmonic polynomials in
$\mathcal{P}(\mathbb{R}^n)$, and $\mathcal{H}_k (\mathbb{R}^n)$ for the subspace of all $k$-homogeneous harmonic polynomials.
Similarly,  we define $\mathcal H_{\leq d}(\mathbb{R}^n)$.
Both $\mathcal{H}(\mathbb{R}^n)$ and $\mathcal{H}_d(\mathbb{R}^n)$ are $\mathcal{O}_n$-invariant, that is, $f \circ O$ belongs to the same space as $f$ for every $O \in \mathcal{O}_n$.

Observe that  
\begin{equation}\label{ludo1} \mathcal H_{\leq d}(\mathbb{R}^n) = \spa_{k \leq d} \mathcal H_{k}(\mathbb{R}^n) \quad \text{and} \quad
\mathcal H(\mathbb{R}^n) = \spa_k \mathcal H_k(\mathbb{R}^n)\,.
\end{equation}
In fact, if $f\in \mathcal{H} (\mathbb{R}^n)$ has degree $d$, then $f= \sum_{k=0}^d f_k$ for some $f_k \in \mathcal{P}_k(\mathbb{R}^n)$. But since $\triangle f_k$ for $k\geq 2$ is $(k-2)$-homogeneous (resp., $0$-homogeneous whenever for $k<2$), it follows by the unique homogeneous decomposition  $\triangle f = \sum_{k=0}^d\triangle f_k =0$  that $\triangle f_k =0$ for each $0 \leq k \leq d$.

Much of what follows is based on the following well-known decomposition of $\mathcal{P}_{d}(\mathbb{R}^n)$ into
harmonic subspaces (see, e.g.\, \cite[Theorem~2.18]{atkinson2012spherical}).

\begin{lemma} \label{harmdim}
For each $d\geq 2$ and  $n\geq 1$ we have the following orthogonal sum{\rm:}
\begin{equation*}
\mathcal{P}_d(\mathbb{R}^n) =\mathcal{H} _d(\mathbb{R}^n) \oplus \|\cdot\|_2^2\,\, \mathcal{P}_{d-2}(\mathbb{R}^n)\,,
\end{equation*}
that is, every $f \in \mathcal{P}_d(\mathbb{R}^n)$
 has a unique decomposition 
\[
f(x ) = g(x) + \|x\|_2^2 \,h(x)\,,\quad x \in \mathbb{R}^n
\]
 with $g \in \mathcal{H}_d(\mathbb{R}^n)$
 and $h \in \mathcal{P}_{d-2}(\mathbb{R}^n)$\,.
\end{lemma}

All restrictions of polynomials  in $\mathcal{H}(\mathbb{R}^n)$ to $\mathbb{S}^{n-1}$ are denoted by
$
\mathcal{H}(\mathbb{S}^{n-1})\,,
$
and those are called spherical harmonics on $\mathbb{S}^{n-1}$.

We write   $\mathcal{H}_d(\mathbb{S}^{n-1})$ for the
space collecting all  restrictions of polynomials from  $\mathcal{H}_d(\mathbb{R}^n)$ to $\mathbb{S}^{n-1}$,
and define $\mathcal{H}_{\leq d}(\mathbb{S}^{n-1})$ similarly.
 Endowed with the supremum norm taken
on $\mathbb{S}^{n-1}$, both spaces 
form finite-dimensional 
$\mathcal{O}_n$-invariant subspaces of the Banach space $C(\mathbb{S}^{n-1})$.

Iterating Lemma~\ref{harmdim}, leads to the following well-known decomposition theorem of spherical harmonics.

\begin{proposition} \label{orthodeco}
We have
\begin{equation}\label{LH9}
\mathcal H_{\leq d} (\mathbb{S}^{n-1})= \bigoplus_{k \leq d}\mathcal{H}_k ( \mathbb{S}^{n-1})
\quad \text{and}  \quad
  \mathcal H ( \mathbb{S}^{n-1})= \bigoplus_k \mathcal{H}_k ( \mathbb{S}^{n-1})\,,
\end{equation}
and
\begin{equation}\label{klm2}
\mathcal P_d(\mathbb{S}^{n-1})
=  \bigoplus_{j=0}^{\lfloor d/2\rfloor} \,\mathcal{H}_{d-2j}(\mathbb{S}^{n-1})\,,
  \end{equation}
  where all  decompositions are orthogonal in $L_2( \mathbb{S}^{n-1})$. In particular,
  \begin{equation}\label{ludoAA}
  \mathcal H_{\leq d} ( \mathbb{S}^{n-1})=\mathcal P_{\leq d} ( \mathbb{S}^{n-1})
    \quad \text{and} \quad
      \mathcal H (\mathbb{S}^{n-1})=\mathcal P ( \mathbb{S}^{n-1})\,.
\end{equation}
\end{proposition}
\begin{proof}
 Looking at~\eqref{ludo1}, we immediately see~\eqref{LH9}. 
Moreover,  since $ \|x\|_2^2=1$
for every  $x \in \mathbb{S}^{n-1} $, we deduce from Lemma~\ref{harmdim} that~\eqref{klm2} holds algebraically.
Together with \eqref{ludoABAB} this implies  \eqref{ludoAA}.
For the simple argument (based on Green's identity) showing that  $\mathcal{H}_{k}(\mathbb{S}^{n-1})$ and $\mathcal{H}_{\ell}(\mathbb{S}^{n-1})$ are  
orthogonal in $L_2( \mathbb{S}^{n-1})$
see, e.g. \cite[Corollary~2.15]{atkinson2012spherical}.
\end{proof}

As an immediate consequence of Lemma~\ref{harmdim} 
and~\eqref{dimformuA}, we see that for all    $n \geq 2$ and $d \geq 1$
\begin{equation} \label{dimfor}
N_{n,d} := \dim \mathcal{H}_d(\mathbb{S}^{n-1})
=  \frac{(n+ 2d-2)(n+d-3)!}{d!(n-2)!}\,.
\end{equation}

\smallskip

The following density result is well-known and fundamental for our purposes.

\smallskip

\begin{theorem} \label{realdense}
$\mathcal H ( \mathbb{S}^{n-1})=\spa_{k}\mathcal{H}_k ( \mathbb{S}^{n-1})$ is dense in $C(\mathbb{S}^{n-1})$.
\end{theorem}

A standard proof proceeds by first applying the Stone-Weierstrass theorem to show that $\mathcal P(\mathbb{S}^{n-1})$ is dense in $C(\mathbb{S}^{n-1})$. The conclusion then follows from~\eqref{ludoAA} and~\eqref{LH9}.

It is imperative to emphasize that all elements within the aforementioned function spaces have values in the complex plane $\mathbb{C}$. In particular, all function spaces considered form complex Banach spaces. For a treatment on real-valued function, we refer to Section \ref{final remark}.

\subsection{Reproducing kernels} \label{la Rudin}

For every closed subspace $S$ of $L_2(\sigma_n)$, we denote by $\pi_S$ the orthogonal projection from $L_2(\sigma_n)$ onto $S$. 
It follows directly that the orthogonal projection $\pi_S$ commutes with the action of $\mathcal{O}_n$ on $C(\mathbb{S}^{n-1})$. Specifically,
\begin{equation}\label{lemma: pi_S commutes with L_V} \pi_S(f\circ A) = \pi_S(f) \circ A, \end{equation}
for every $f \in C(\mathbb{S}^{n-1})$ and $A \in \mathcal{O}_n$. We call this property $\mathcal O_n$-equivariance.
\smallskip

\begin{lemma} \label{lemm: reproducing}
 For every finite-dimensional subspace $S \subset C(\mathbb{S}^{n-1})$, there exists a unique continuous function 
\[
\mathbf{k}_S : \mathbb{S}^{n-1} \times \mathbb{S}^{n-1} \to \mathbb{C},
\]
the reproducing kernel of $S$, satisfying the following properties:

\begin{itemize}
    \item[(i)] For all \( f \in L_2(\sigma_n) \) and \( x \in \mathbb{S}^{n-1} \),
    \[
    (\pi_S f)(x) = \big\langle f, \mathbf{k}_S(x, \cdot) \big\rangle_{L_2(\sigma_n)}.
    \]
    \item[(ii)] For every \( x \in \mathbb{S}^{n-1} \), the function \( \mathbf{k}_S(x, \cdot) \) belongs to \( S \).
    \item[(iii)] The function \( \mathbf{k}_S \) is Hermitian, meaning that
    \[
    \mathbf{k}_S(x, y) = \overline{\mathbf{k}_S(y, x)}, \quad\, x,y \in \mathbb{S}^{n-1}\,.
    \]  
\end{itemize}

Moreover, if \( S \) is invariant under the action of \( \mathcal{O}_n \), then the following additional properties hold:

\begin{itemize}
    \item[(iv)] The function \( \mathbf{k}_S \) is rotation-invariant, that is,
    \[
    \mathbf{k}_S(Ax, Ay) = \mathbf{k}_S(x, y), \quad\, A \in \mathcal{O}_{n}, \, \, \, \,  x, y \in \mathbb{S}^{n-1}\,.
    \]
    \item[(v)] The diagonal values of $\mathbf{k}_S$ satisfy
    \[
    \mathbf{k}_S(x, x) = \dim S, \quad\, x \in \mathbb{S}^{n-1}\,.
    \]
    \item[(vi)] Additionally, the following integral identity  hold:
    \[
    \left( \int_{\mathbb{S}^{n-1}} |\mathbf{k}_S(x, y)|^2 \, d\sigma_n(y) \right)^{\frac{1}{2}} = \sqrt{\dim S}, \quad\, x \in \mathbb{S}^{n-1}\,.
    \]
\end{itemize}
\end{lemma}
To establish these desired properties, we just choose an orthonormal basis \( (f_j)_{j=1}^{\dim S} \) of \( S \) and extend it to an orthonormal basis of \( L_2(\sigma_n) \). So we can explicitly describe the orthogonal projection \( \pi_S \) onto~\( S \):  
\[
\pi_S f = \sum_{j=1}^{\dim S} \big\langle f, f_j \big\rangle_{L_2(\sigma_n)} f_{j}\,, 
\quad\, f \in L_2(\sigma_n)\,.
\]
Evaluating this expression at \( x \in \mathbb{S}^{n-1} \), we obtain  
\[
\pi_S f(x) =
\sum_{j=1}^{\dim S} 
\Big(\int_{\mathbb{S}^{n-1}} f(y) \overline{f_j(y)}\, d\sigma_n(y)\Big) f_j(x)
=
\int_{\mathbb{S}^{n-1}} f(y) \,
\Big(\sum_{j=1}^{\dim S} 
\overline{f_j(y)} f_j(x)  \Big)\, d\sigma_n(y).
\]
Thus the function \( \mathbf{k}_S: \mathbb{S}^{n-1} \times \mathbb{S}^{n-1} \to \mathbb{C} \) is given by  
\begin{align} \label{ker-def}
  \mathbf{k}_S(x,y) := \sum_{j=1}^{\dim S} \overline{f_j(x)} f_j(y),   \quad x,y\in \mathbb{S}^{n-1}.
\end{align}
With this definition, it is immediate that \( \mathbf{k}_S \) satisfies (i).  
To verify the remaining properties, we just need to perform routine calculations which are left to the reader.

If $S$ is invariant under the action of the orthogonal group, then the norm of the orthogonal projection onto $S$ can be computed in terms of the $L_1$-norm of its reproducing kernel. As usual, we denote by $e_1$ the canonical basis vector with 1 in the first coordinate and 0 elsewhere.

\begin{remark}  
\label{main-ibkA}
Let $S$ be an $\mathcal O_n$-invariant, finite-dimensional subspace $S$ of  $C(\mathbb{S}^{n-1})$. Then we have
\begin{equation} 
 \big\|\pi_S:C(\mathbb{S}^{n-1}) \to S\big\|=\int_{\mathbb{S}^{n-1} }|\mathbf{k}_S(e_1,y) |\, d\sigma_n(y) \,.
\end{equation}
\end{remark}

\begin{proof}
By property $(i)$ in Lemma~\ref{lemm: reproducing} and the Riesz representation theorem
we have that
\begin{align*}
\big\|\pi_S:C(\mathbb{S}^{n-1}) \to S\big\|
&
=
\sup_{\|f\|_\infty \leq  1} \sup_{x \in \mathbb{S}^{n-1}}
\Big|   \int_{\mathbb{S}^{n-1} } f(y) \overline{\mathbf{k}_S(x,y)}\, d \sigma_n(y) \Big|
\\&
=
\sup_{x \in \mathbb{S}^{n-1}}
\sup_{\|f\|_\infty \leq  1}
\Big|   \int_{\mathbb{S}^{n-1} } f(y) \overline{\mathbf{k}_S(x,y)}\,d \sigma_n(y) \Big|
=
\sup_{x \in \mathbb{S}^{n-1}}
\int_{\mathbb{S}^{n-1} }|\mathbf{k}_S(x,y) |\,d\sigma_n(y)\,.
\end{align*}
Then we use property $(iv)$  and the rotation invariance  of the surface measure $\sigma_n$ to see that 
\begin{equation*}
\int_{\mathbb{S}^{n-1}}|\mathbf{k}_S(x,y) |\,d\sigma_n(y)
=
\int_{\mathbb{S}^{n-1} }|\mathbf{k}_S(e_1,y) |\,d\sigma_n(y), \quad\, x\in \mathbb{S}^{n-1}\,. 
\qedhere
\end{equation*}
\end{proof}

\smallskip

\section{Projection constants: abstract part}
We will focus on finite-dimensional subspaces \( S \) of \( C(\mathbb{S}^{n-1}) \) for which the restriction of the orthogonal projection \( \pi_S \) on \( C(\mathbb{S}^{n-1}) \) is the unique \( \mathcal{O}_n \)-equivariant projection. This means that \( \pi_S \) is the unique projection \( \mathbf{Q} \) on \( C(\mathbb{S}^{n-1}) \) onto \( S \) that satisfies the condition
\[
\mathbf{Q} (f \circ A) = \mathbf{Q}(f) \circ A, \quad\, f \in S, \, \,\, A \in \mathcal{O}_{n}\,.
\]

\subsection{Averaging projections} \label{access}
The following result is one of our main abstract tools. It shows how to compute the projection constant of a subspace $S$ in terms of the $L_1$-norm of  the kernel $\mathbf{k}_S$. 

\begin{theorem} \label{main-ibk}
Let $S$  be an $\mathcal{O}_n$-invariant, finite-dimensional  subspace of $C(\mathbb{S}^{n-1})$ such that $\pi_S$ is the unique $\mathcal O_n$-equivariant projection from $C(\mathbb{S}^{n-1})$ onto $S$. Then
\[
\boldsymbol{\lambda}(S) =  \big\|\pi_S:C(\mathbb{S}^{n-1}) \to S\big\|=\int_{\mathbb{S}^{n-1} }|\mathbf{k}_S(e_1,y) |\,d\sigma_n(y) \,.
\]
\end{theorem}

The proof uses the standard technique of averaging projections - see, e.g.~\cite{rudin1962projections} or \cite[Theorem III.B.13]{wojtaszczyk1996banach}.

\begin{proof}
Take any projection $\mathbf{Q}$ of $C(\mathbb{S}^{n-1})$ onto $S$, and for every $A\in \mathcal O_n$ consider the composition operator    $\phi_A \in \mathcal{L}(C(\mathbb{S}^{n-1}))$ given by 
\[
\phi_A(f) := f\circ A, \quad\, f\in C(\mathbb{S}^{n-1})\,.
\]
Clearly, $\phi_{id}$ is the identity on $C(\mathbb{S}^{n-1})$ (where $id$ denotes the unity in $\mathcal O_n$), and $\phi_{AB} = \phi_B \phi_A$ for all $A, B \in \mathcal O_n$. Furthermore,
\begin{equation*}
    \mathcal O_n \ni A \mapsto \phi_{A^{-1}}\mathbf{Q} \, \phi_A \in \mathcal{L}(C(\mathbb{S}^{n-1}))
\end{equation*}
is Bochner integrable, as it is bounded by the Banach-Steinhaus theorem and measurable due to its weak measurability. Then the Bochner integral
\begin{equation*}\label{equation rudy}
\int_{\mathcal O_n} \phi_{A^{-1}}\mathbf{Q}\,\phi_A\,d\textbf{m}(A)
\in \mathcal{L}(C(\mathbb{S}^{n-1})),
\end{equation*} where $\textbf{m}$ is the normalized Haar measure on $\mathcal{O}_n$, 
is a projection from $C(\mathbb{S}^{n-1})$ onto $S$, which is $\mathcal O_n$-equivariant. By assumption $\pi_S$ is the unique projection with this property, so we have that 
\begin{equation*}
\pi_S= \int_{\mathcal O_n} \phi_{A^{-1}}\mathbf{Q} \,\phi_A\,d\textbf{m}(A)\,.
\end{equation*}
Consequently,  the first equality of the statement follows by the triangle inequality (for Bochner integrals) and the second is just
a~consequence of Remark~\ref{main-ibkA}.
\end{proof}

The next statement shows that if several orthogonal subspaces each admit a unique $\mathcal{O}_n$-equivariant projection, then their direct sum also admits a unique $\mathcal{O}_n$-equivariant projection. 

\begin{proposition} \label{sum of accessible}
Let $\{S_k\}_{k=1}^r$ be finitely many  
 $\mathcal{O}_n$-invariant, finite-dimensional subspaces of $C(\mathbb{S}^{n-1})$ which are \begin{enumerate}
    \item[i)]  pairwise $L_2(\sigma_n)$-orthogonal, and such that
    \item[ii)] each $\pi_{S_k}\colon C(\mathbb{S}^{n-1}) \to S_k$ is the unique $\mathcal{O}_n$-equivariant projection.
\end{enumerate}
Then for $S = \bigoplus_{k=1}^r S_k$, the following hold{\rm:}
\begin{enumerate}
    \item The projection $\pi_S = \sum_{k=1}^r \pi_{S_k}\colon C(\mathbb{S}^{n-1}) \to S$ is the unique $\mathcal{O}_n$-equivariant projection.
    \item The reproducing kernel of $S$ decomposes as $\mathbf{k}_S(x,y) = \sum_{k=1}^r \mathbf{k}_{S_k}(x,y)$, for all $x,y\in \mathbb{S}^{n-1}$.
\end{enumerate}
\end{proposition}

\begin{proof}
By induction it suffices to show this for $r=2$. 
That $S_1 \oplus S_2$ is  $\mathcal O_n$-invariant  is straightforward.
Observe that $\pi_{S_1 \oplus S_2} = \pi_{S_1} + \pi_{S_2}$
defines the orthogonal projection from $L_2(\sigma_n)$ onto $S_1 \oplus S_2$.
Using the properties of reproducing kernels, we obtain that for every $f \in L_2(\sigma_n)$ and $x \in K$,
\[
(\pi_{S_1 \oplus S_2}f)(x) = (\pi_{S_1}f)(x) + (\pi_{S_2}f)(x) =
\int_{\mathbb{S}^{n-1}}f(y)\overline{(\mathbf{k}_{S_1}(x, y)+\mathbf{k}_{S_2}(x, y))} \, d\sigma_n(y)
\,.
\]
Hence by the uniqueness of reproducing kernel  we get
\[
\mathbf{k}_{S_1+S_2}(x, y)
 \,=\,
\mathbf{k}_{S_1}(x, y) +  \mathbf{k}_{S_2}(x, y)\,, \quad\, x,y\in \mathbb{S}^{n-1}\,.
\]
Let us now show that $\pi_S$ is the unique $\mathcal O_n$-equivariant projection. So let $\mathbf{Q}$ be a projection from
$C(\mathbb{S}^{n-1})$ onto $S=S_1 \oplus S_2$ which commutes with the action of $\mathcal O_n$ on
$C(\mathbb{S}^{n-1})$. We claim that $\mathbf{Q} = \pi_{S_1 \oplus S_2}$. Indeed, consider the two projections
\[
\text{$\mathbf{Q}_{S_1} = \pi_{S_1}\circ   \mathbf{Q} $ \quad and \quad $\mathbf{Q}_{S_2} = \pi_{S_2}\circ   \mathbf{Q} $}
\]
from $C(\mathbb{S}^{n-1})$ onto $S_1$ and ${S_2}$, respectively.
Since the projections $\pi_{S_1}$ and $\pi_{S_2}$ on $C(\mathbb{S}^{n-1})$ are $\mathcal{O}_n$-equivariant, the same holds for $\mathbf{Q}_{S_1}$ and $\mathbf{Q}_{S_2}$. Then by uniqueness properties of the projections on $S_1$ and $S_2$ we see that
\[
\text{$\mathbf{Q}_{S_1} = \pi_{S_1}$ \quad and \quad $\mathbf{Q}_{S_2} = \pi_{S_2}$}\,,
\]
and hence for all $f \in C(\mathbb{S}^{n-1})$ as desired
\[
\mathbf{Q}f = \pi_{S_1}(\mathbf{Q}f)  + \pi_{S_2}(\mathbf{Q}f) 
= \mathbf{Q}_{S_1}f + \mathbf{Q}_{S_2}f
= \pi_{S_1}f + \pi_{S_2}f 
= \pi_{S_1 \oplus S_2} f
\,.
\]
This concludes the proof.
\end{proof}

We would like to acknowledge the anonymous referee for suggesting a simpler argument, which leads directly to the following uniqueness statement.

\begin{proposition}\label{cor:unique_projection}
The orthogonal projection 
$\pi_{\mathcal{H}_d(\mathbb{S}^{n-1})} \colon
C(\mathbb{S}^{n-1}) \to \mathcal{H}_d(\mathbb{S}^{n-1})$
is the unique $\mathcal O_n$-equi\-variant projection 
from $C(\mathbb{S}^{n-1})$ onto $\mathcal{H}_d(\mathbb{S}^{n-1})$.

The analogous statement also holds if we replace $\mathcal{H}_d(\mathbb{S}^{n-1})$ by $\mathcal P_d(\mathbb{S}^{n-1})$ or by $\mathcal P_{\leq d}(\mathbb{S}^{n-1})$.
\end{proposition}

\begin{proof}
\noindent Take a  $\mathcal O_n$-equivariant projection 
$\mathbf{Q}$ from $C(\mathbb{S}^{n-1})$ onto $\mathcal{H}_d(\mathbb{S}^{n-1})$.  By Theorem~\ref{realdense} and Proposition~\ref{orthodeco} we know that $$C(\mathbb{S}^{n-1}) = \overline{\text{span}}_k \,\mathcal{H}_k(\mathbb{S}^{n-1})$$ is a decomposition of $C(\mathbb{S}^{n-1})$ into the $\mathcal{O}_n$-invariant, pairwise $L_2(\sigma_{n})$-orthogonal subspaces $\mathcal{H}_k(\mathbb{S}^{n-1})$. Hence it suffices to check that
\[
\mathbf{Q}\restrict{\mathcal{H}_k(\mathbb{S}^{n-1})}  = 0,  \quad\,  k \neq d\,.
\]
Take $f \in \mathcal{H}_k(\mathbb{S}^{n-1}),\, k \neq d$, and fix some orthonormal basis $(f_i)$ of $\mathcal{H}_k(\mathbb{S}^{n-1})$.
By the $\mathcal O_n$-invariance of $\mathcal{H}_k(\mathbb{S}^{n-1})$, for every  $A \in \mathcal{O}_n$ we have that
\[
f \circ A^{-1}  = \sum_i \,\left(\int_{\mathbb{S}^{n-1}}  (f \circ A^{-1})(y)  \overline{f_i(y)} d\sigma_{n}(y) \,\right)\, f_i\,,
\]
and hence  we conclude from the $\mathcal O_n$-equivariance of $\textbf{Q}$ that for every $x \in \mathbb{S}^{n-1}$
\begin{align*}
  \mathbf{Q}(f)(x)
    =  \mathbf{Q}(f \circ A^{-1})(Ax)
      =\sum_i \,\left(\int_{\mathbb{S}^{n-1}}   (f \circ A^{-1})(y)  \overline{f_i(y)} d\sigma_{n}(y) \,\right)\, \mathbf{Q}(f_i)(Ax)\,.
\end{align*}
Then by the rotation invariance of the surface measure
\begin{align*}
  \mathbf{Q}(f)(x)
     =
  \int_{\mathbb{S}^{n-1}}   f (y) \,\, \bigg[\sum_i   \overline{f_i(Ay)} \mathbf{Q}(f_i)(Ax)\bigg]\, d\sigma_{n}(y) \,.
\end{align*}
Integrating with respect to the Haar measure on $\mathcal{O}_n$ and using Fubini's theorem, we conclude that for all
$x \in \mathbb{S}^{n-1}$
\[
\mathbf{Q}(f)(x)
     =
  \int_{\mathbb{S}^{n-1}}   f (y) \,\,   \bigg[\int_{\mathcal{O}_n}\sum_i   \overline{f_i(Ay)}  \mathbf{Q}(f_i)(Ax)
  d\textbf{m}(A)\bigg]\,
   d\sigma_{n}(y)\,.
\]

It is worth noting that for every pair  fixed distinct points $x, y \in \mathbb{S}^{n-1}$, there exists an orthogonal transformation $B \in \mathcal{O}_n$ such that $Bx = y$ and $By = x$. One such transformation is the reflection across the hyperplane orthogonal to $x - y$, which is given explicitly by
\[
B = \id - 2\, \frac{\langle \cdot , x - y \rangle}{\|x - y\|_2^2} (x - y).
\]
This linear operator fixes pointwise the orthogonal complement of $\operatorname{span}\{x - y\}$ and satisfies $B^2 = \id$, so it is a symmetric involution that exchanges $x$ and $y$.

Fixing $x,y \in \mathbb{S}^{n-1}$ we make the change of integration variable $A \in \mathcal{O}_n \leftrightarrow A \circ B \in \mathcal{O}_n$. Thus, using  the rotation invariance  of the Haar measure on $\mathcal{O}_n$
and again  the $\mathcal O_n$-equivariance of $\mathbf{Q}$, we see that for all
$x \in \mathbb{S}^{n-1}$
\[
\mathbf{Q}(f)(x)
     =
  \int_{\mathbb{S}^{n-1}}   f (y) \, g(x,y)
   d\sigma_{n}(y)\,,
\]
where
\[
g(x,y) := \int_{\mathcal{O}_n}\sum_i   \overline{f_i(Ax)} \,\mathbf{Q}(f_i \circ A)(y)
  d\textbf{m}(A)\,,\quad x,y \in  \mathbb{S}^{n-1}\,.
\]
Then
$g(x,\cdot) \in \mathcal{H}_d(\mathbb{S}^{n-1})$ for all
$x \in \mathbb{S}^{n-1}$,
since $\mathbf{Q}(f_i \circ A)\in \mathcal{H}_d(\mathbb{S}^{n-1})$ for all possible  $i,A$
(to see this interpret the preceding integral as a Bochner integral of a~function with values in $\mathcal{H}_d(\mathbb{S}^{n-1})$). But on the other hand,  $f \in \mathcal{H}_k(\mathbb{S}^{n-1})$
with $k \neq d$, so that by orthogonality $\mathbf{Q}(f)=0$.
This shows that $\textbf{Q}$ coincides with $\pi_{\mathcal{H}_d(\mathbb{S}^{n-1})}$.

The analogous statements for $\mathcal P_d(\mathbb{S}^{n-1})$ or  $\mathcal P_{\leq d}(\mathbb{S}^{n-1})$, follow from a direct use of Propositions \ref{orthodeco} and \ref{sum of accessible}.
\qedhere
\end{proof}

\subsection{Axially invariance}
Our aim now is to find in Theorem~\ref{main-ibk} more concrete 
descriptions of the  kernels $\mathbf{k}_S(e_1,\cdot)$, whenever $S$ is one of the  three 
spaces 
$\mathcal{H}_{ d}(\mathbb{S}^{n-1})$,
$\mathcal{P}_{d}(\mathbb{S}^{n-1})$
and $\mathcal{P}_{\leq d}(\mathbb{S}^{n-1})$.

Although the content of this subsection is essentially familiar to the approximation theory community, we choose to provide full proofs. However, it is important to emphasize that we cannot directly cite the results we need, and in many cases, the proofs have been adapted or require some minor modifications. For this reason, we believe that offering a comprehensive overview is necessary, avoiding the reader's need to constantly refer to other texts to fill in the details.

First, we need to establish some notation. Let $\mathcal{O}_n(e_1)$ denote the subgroup of orthogonal transformations that fix $e_1$; that is, $A \in \mathcal{O}_n(e_1)$ if  $A$ is orthogonal and $Ae_1 = e_1$. We consider functions
$f:\mathbb{R}^{n} \to \mathbb{C}$
that are invariant under the action of this subgroup, meaning that $f \circ A = f$ for every $A \in \mathcal{O}_n(e_1)$. We will refer to such functions, somewhat informally, as axially invariant.

Given $d \in \mathbb{N}_0$, we demonstrate the existence and uniqueness of  harmonic, axially invariant polynomials  (see also
 \cite[Section~2.1.2]{atkinson2012spherical}).

\smallskip
 
\begin{proposition} \label{propleg}
There is a unique  polynomial
$f \in \mathcal{H}_d(\mathbb{R}^n)$ which fulfills the following two properties{\rm:}
\begin{itemize}
\item[(a)]
$f\circ A = f$ for all $A\in \mathcal{O}_n(e_1)$\,,
\item[(b)]
$f(e_1) = 1$\,.
\end{itemize}
More precisely,  this unique polynomial is given by 
\begin{equation*}
L_{n,d}(x) = \sum_{j=0}^{\lfloor d/2\rfloor} b_j(n,d)  x_1^{d-2j}  \Vert(x_2, \ldots, x_{n})\Vert_2^{2j},
\end{equation*}
where the coefficients $b_j(n,d)$ are determined by{\rm:}
\begin{equation}\label{unique-harm}
b_j(n,d) = \frac{(-1)^j d!\, \Gamma\Big( \frac{n-1}{2}\Big)}{4^j j!\,(d-2j)!\, \Gamma\big( j+\frac{n-1}{2}\big)}.
\end{equation}
\end{proposition}

Sometimes, we refer 
 to $L_{n,d} \colon \mathbb{R}^n \to \mathbb{R}$ as the Legendre harmonic of degree $d$ in $n$ variables.

\begin{proof}
Take $f$ verifying both properties and choose a representation
\[
f(x) = \sum_{\substack{\alpha \in \mathbb{N}^n_0\\ |\alpha|=d}} c_\alpha x^\alpha, \quad x \in \mathbb{R}^n\,.
\]
We write $x = (x_1,x') \in \mathbb{R} \times \mathbb{R}^{n-1}$ for  $x \in \mathbb{R}^n$, and note that
every $n\times n$-matrix $A$ belongs to $\mathcal{O}_n(e_1)$ if and only if it has the form
\begin{equation}\label{0a0}
  A=
\begin{pmatrix}
  1 & 0  \\
  0 & A' \\
 \end{pmatrix}
\end{equation}
with $A' \in \mathcal{O}_{n-1}$. We fix $x_1 \in \mathbb{R}$,  and consider the polynomial
\[
g(x') := f(x_1, x'),  \quad x' \in \mathbb{R}^{n-1}.
\]
Then for all $A' \in \mathcal{O}_{n-1}$ and $x' \in \mathbb{R}^{n-1}$
\[
 g(A'x')= f(x_1,A'x' ) = f(A(x_1,x') ) = f(x_1,x' )=  g(x')\,,
\]
and in particular $g(x') = g(-x')$. Consequently,
\[
g(t, 0, \ldots, 0) = \sum_{j=0}^{\lfloor d/2\rfloor} a_j t^{2j},  \quad t \in \mathbb{R}\,,
\]
where the complex coefficients $a_j= a_j(x_1)$ are functions in $x_1$.
Choosing for $x' \in \mathbb{R}^{n-1}$ some $A' \in \mathcal{O}_{n-1}$ such that $A' x' = \|x'\|_2 e_1$,
we obtain
\[
\sum_{|\alpha|=d} c_\alpha x_1^{\alpha_1} (x')^{(\alpha_2, \ldots, \alpha_{n-1})} =Q(x_1,x') = g(x') = g(\|x'\|_2 e_1) = \sum_{j=0}^{\lfloor d/2\rfloor} a_j(x_1) \|x'\|_2^{2j}\,.
\]
 If we  now compare coefficients,  then for all $x \in \mathbb{R}^n$
\[
f(x) = \sum_{j=0}^{\lfloor d/2\rfloor} b_j x_1^{d-2j} \|x'\|_2^{2j}\,.
\]
Finally, it remains to verify that  the coefficients $b_j = b_j(n,d)$  are those from \eqref{unique-harm}.  Indeed, using that $f$ is harmonic, a careful computation shows that for all $x \in \mathbb{R}^n$ 
\[
0 =\triangle f(x) =
\sum_{j=0}^{\lfloor d/2\rfloor}
\big[(d-2j)(d-2j-1) b_j + (2n-2 +4j)(j+1)b_{j+1}\big]
x_1^{d-2-2j} \|x'\|_2^{2j}\,.
\]
Then for $j = 1, \ldots, \lfloor d/2\rfloor  $
\[
b_j =  -  \frac{(d-2j+2)(d-2j+1)}{2j (2j+n-3) } b_{j-1}\,,
\]
and since $b_0 =1$, we obtain the  desired  claim.

Conversely, let us explain why the polynomial  $L_{n,d}$ has the desired properties. Clearly it is $d$-homogeneous, $L_{n,d}(e_1)=1$,
and the last calculations  prove  its  harmonicity. Moreover, looking at~\eqref{0a0} also property $(a)$ is clear.
  \end{proof}

In order to derive a few important consequences of the preceding proposition, we consider the equality
\begin{equation} \label{legent}
 L_{n,d}^{ \diamond}(\langle x, e_1\rangle) =  L_{n,d}(x)  \,, \quad \,x \in \mathbb{S}^{n-1}\,,
\end{equation}
where $L_{n,d}^{ \diamond}\colon [-1,1]\to \mathbb{R} $   is given by
\begin{equation} \label{eq: Legendre polynomial}
  L_{n,d}^{ \diamond}(t) :=
    \sum_{j=0}^{\lfloor d/2\rfloor} b_j(n,d) \,t^{d-2j}\,    (1-t^2)^{j}\,, \quad\, t\in [-1, 1]\,.
 \end{equation}
This  polynomial  is referred to as the 'Legendre polynomial of degree $d$ in $n$ dimensions'. The subsequent consequence of Proposition~\ref{propleg} is adapted from \cite[Theorem~2.8]{atkinson2012spherical}.

\begin{corollary} \label{corollary1}
Let $f \in \mathcal{H}_d(\mathbb{S}^{n-1})$. Then the following are
equivalent{\rm:}
\begin{itemize}
\item[(1)]
$f \circ A = f$ for all $A \in \mathcal{O}_n(e_1)$\,,
\item[(2)]
$f = f(e_1) L_{n,d}^{ \diamond}(\langle\pmb{\cdot},e_1\rangle)$\,.
\end{itemize}
\end{corollary}

\begin{proof}
Using the properties of $L_{n,d}$ isolated in Proposition~\ref{propleg}, the implication $(2)\Rightarrow(1)$ is obvious. In order to  check that  $(1)\Rightarrow(2)$, choose $g \in \mathcal{H}_d(\mathbb{R}^n)$ such that $g|_{\mathbb{S}^{n-1}} = f$. By assumption for each $A \in \mathcal{O}_n(e_1)$ and $x \in \mathbb{R}^n$
\[
g(Ax) = \|Ax\|^d_2 f \Big(\frac{Ax}{\|Ax\|_2}\Big)
= \|x\|^d_2 (f\circ A) \Big(\frac{x}{\|x\|_2}\Big) 
=  \|x\|^d_2 f \Big(\frac{x}{\|x\|_2}\Big) = g(x)\,.
\]
Hence by the 'uniqueness properties' of the Legendre harmonic $L_{n,d}$ from Proposition~\ref{propleg} there is some constant $c >0$ such that
$g = c L_{n,d}$ on $\mathbb{R}^n$, and inserting here $e_{1}$, we see that $c = g(e_1)=f(e_1)$. Then for all $\eta \in \mathbb{S}^{n-1}$, we get
\begin{equation*}
f(\eta) = g( \eta) = c L_{n,d}(\eta) = f(e_1) L_{n,d}^{ \diamond}(\langle\eta, e_1\rangle)\,. \qedhere
\end{equation*}
\end{proof}

 \smallskip

As promised,  we now consider  the reproducing kernels of the spaces $\mathcal{H}_{d}(\mathbb{S}^{n-1})$, $\mathcal{P}_{d}(\mathbb{S}^{n-1})$, and $\mathcal{P}_{\leq d}(\mathbb{S}^{n-1})$. The results presented are consequences of the fact that the kernel $ \textbf{k}_{\mathcal{H}_{ d}(\mathbb{S}^{n-1})}(e_1, \cdot)$ is axially invariant. 

\begin{proposition}
Let  $n \geq 2$ and $d \geq 1$. Then the following identities hold:
\begin{equation}\label{ker: un-oX}
 \textbf{k}_{\mathcal{H}_{ d}(\mathbb{S}^{n-1})}(e_1, \cdot) = N_{n,d} \,\,L_{n,d}^{ \diamond}(\langle \cdot, e_1 \rangle),
\end{equation}
\begin{equation}\label{ker: due}
 \textbf{k}_{\mathcal{P}_{d}(\mathbb{S}^{n-1})}(e_1, \cdot) = 
 \sum_{j=0}^{\lfloor d/2\rfloor} N_{n,d-2j} \,\,L_{n,d-2j}^{\diamond}(\langle \cdot, e_1 \rangle),
\end{equation}
\begin{equation}\label{ker: tres}
\textbf{k}_{\mathcal{P}_{\leq d}(\mathbb{S}^{n-1})}(e_1, \cdot) = 
\sum_{j=0}^{d} N_{n,j} \,\,L_{n,j}^{\diamond}(\langle \cdot, e_1 \rangle),
\end{equation}    
where $N_{n,k}$ denotes the dimension of $ \mathcal{H}_k(\mathbb{S}^{n-1})$ as defined in \eqref{dimfor}.
\end{proposition}
                               
\begin{proof}
The proof of Equation~\eqref{ker: un-oX} follows from Corollary~\ref{corollary1} and the rotation-invariance of the reproducing kernel (see Lemma~\ref{lemm: reproducing}, property (iv)).
To show the equalities given in Equations \eqref{ker: due} and \eqref{ker: tres} we simply use Propositions~\ref{orthodeco} and \ref{sum of accessible} and the proven description of the kernel in Equation \eqref{ker: un-oX}. 
\end{proof}

\subsection{Integral representations}

As announced our main intention in this article is to collect  information on the projection constants of
$\mathcal{H}_{ d}(\mathbb{S}^{n-1})$,
$\mathcal{P}_{d}(\mathbb{S}^{n-1})$
and $\mathcal{P}_{\leq d}(\mathbb{S}^{n-1})$. 
The following theorem follows directly from Theorem~\ref{main-ibk}
and Proposition~\ref{cor:unique_projection}.

\begin{theorem} \label{maino}
Let  $n \geq 2$ and $d \geq 1$. Then the following integral formulas hold:
\begin{equation}\label{un-oX}
  \boldsymbol{\lambda}\big(\mathcal{H}_{ d}(\mathbb{S}^{n-1})\big) = N_{n,d}\,\int_{\mathbb{S}^{n-1}}\,
\big|  \,\,L_{n,d}^{ \diamond}(\eta_1)\big| \,d\sigma_n(\eta),
\end{equation}
\begin{equation}\label{due}
 \boldsymbol{\lambda}\big(\mathcal{P}_{d}(\mathbb{S}^{n-1})\big)
 =
\int_{\mathbb{S}^{n-1}}\,
\Big| \sum_{j=0}^{\lfloor d/2\rfloor} N_{n,d-2j} \,\,L_{n,d-2j}^{\diamond}(\eta_1)\Big| \,d\sigma_n(\eta)\,,
\end{equation}
\begin{equation}\label{tres}
\boldsymbol{\lambda}\big(\mathcal{P}_{\leq d}(\mathbb{S}^{n-1})\big)
 = \int_{\mathbb{S}^{n-1}}\,
\Big| \sum_{j=0}^{d} N_{n,j} \,\,L_{n,j}^{\diamond}(\eta_1)\Big| \,d\sigma_n(\eta)\,.
\end{equation}
  \end{theorem}
When $d=1$, the outcome of Equation~\eqref{due} is consistent with Equation~\eqref{beginning} for the real Hilbert space~$\ell_2^n(\mathbb{R})$; see also  the concluding remark in Section~\ref{final remark}.

\section{Projection constants: concrete part}

In this final section, we rewrite the previous integral representations of the projection constants $\mathcal{H}_d(\mathbb{S}^{n-1})$, $\mathcal{P}_d(\mathbb{S}^{n-1})$, and $\mathcal{P}_{\leq d}(\mathbb{S}^{n-1})$ in terms of classical orthogonal polynomials. This reformulation enables us to uncover several of their hidden properties.

To do so, we need a few basic preliminaries on Jacobi polynomials.
They 
 constitute an important  rather wide class
of orthogonal polynomials, from which Chebyshev, Legendre and Gegenbauer polynomials follow as special cases. The main aim here is to relate them
with the Legendre polynomials $L_{n,d}^{ \diamond}$ defined above in  \eqref{eq: Legendre polynomial}.

We start with an observation that will be useful for our purposes. Integrating in spherical coordinates together with the  variable change  $t= \cos \theta$ and Equation
\eqref{omega}, it can be deduced that (see e.g., \cite[Proposition 9.1.2]{faraut2008analysis})
for every $f \in C[-1,1]$ we have
\begin{align}\label{umrechnen}
\begin{split}
  \int_{\mathbb{S}^{n-1}}
f(\eta_1) d\sigma_n(\eta)
&
=
\frac{1}{\pmb{\omega}_n}
\,\,
\int_{\mathbb{S}^{n-1}}
f(\eta_1) ds_n(\eta)
\\&
=
\frac{\pmb{\omega}_{n-1}}{\pmb{\omega}_n}
\,\, \int_{-1}^1 f(t) \; (1-t^2)^{\frac{n-3}{2}} dt
\\&
=
\frac{\Gamma(\frac{n}{2})}{\sqrt{\pi}\, \Gamma(\frac{n-1}{2})}
\,\, \int_{-1}^1 f(t) \; (1-t^2)^{\frac{n-3}{2}} dt \,.
\end{split}
  \end{align}

\smallskip
\noindent
{\bf Jacobi polynomials.} Let $\alpha, \beta  > -1$.  Then we denote the sequence of degree-$d$ Jacobi polynomials on the interval $[-1,1]$ by   $\big(P_d^{(\alpha, \beta)}\big)_{d \in \mathbb{N}_0}$. They are determined by their degree and the orthogonality with respect to the inner product
\[
\langle P,Q\rangle = \int_{-1}^{1} P(t)Q(t)\, (1-t)^\alpha  (1+t)^\beta dt\,,
\]
and normalized by 
\[
P_d^{(\alpha, \beta)}(1) = \binom{d+\alpha}{d}\,.
\]
There are various concrete  ways of writing such polynomials -- for example by the so-called Rodrigues formula
\begin{equation}\label{rodjac}
P_d^{(\alpha, \beta)}(t) = \frac{(-1)^d}{2^d d!}
(1-t)^{-\alpha} (1+t)^{-\beta}
\Big(\frac{\partial}{\partial t}\Big)^d \big[(1-t)^{\alpha+d} (1+t)^{\beta+d}\big], \quad\, t\in [-1, 1]\,.
\end{equation}

\smallskip
\noindent
{\bf Gegenbauer  polynomials.}
Fixing $\lambda>-1/2$, Gegenbauer polynomials are orthogonal on the interval $[-1,1]$ with respect to the weight function $(1-t^2)^{\lambda-\frac{1}{2}}$.
As usual, we take the following Rodrigues type formula as definition:
\begin{equation}\label{RodriguezGegenbauer}
C_d^{(\lambda)}(t)=\frac{(-1)^d}{2^d \;\; d!} \frac{\Gamma(\lambda + \frac{1}{2}) \;\; \Gamma(d+2\lambda)}{\Gamma(2\lambda) \;\; \Gamma(\lambda+n+\frac{1}{2})} (1-t^2)^{-\lambda+1/2} \Big(\frac{\partial}{\partial t}\Big)^d [(1-t^2)^{d+\lambda-\frac{1}{2}}], \quad\, t\in [-1, 1]\,.
\end{equation}
So, up to a multiplicative factor,  Gegenbauer polynomials $C_d^{(\lambda)}$ are the Jacobi polynomials $ P_d^{(\lambda- \frac{1}{2},\lambda- \frac{1}{2})}$,
more precisely,
\begin{align}\label{gegjac}
C_d^{(\lambda)}(t) = \frac{\Gamma(2\lambda +d)\Gamma(\lambda 
+\frac{1}{2})}{\Gamma(2\lambda)\Gamma(\lambda +\frac{1}{2}+d)}
\,P_d^{(\lambda- \frac{1}{2},\lambda- \frac{1}{2})}(t)\,.
\end{align}
We are going to need the following orthogonality relation: 
\begin{equation}\label{Gegenbauerpoly}
\int_{-1}^1 C_{d}^{(\alpha)}(t) \;\; C_{k}^{(\alpha)}(t) \;\;(1-t^2)^{\alpha-\frac{1}{2}} \;\;dt = \frac{\pi\;\; 2^{1-\alpha} \;\;\Gamma(d+2\alpha)}{d!\;\;(d+\alpha)\;\;[\Gamma(\alpha)]^2} \, \delta_{d,k}\,.
\end{equation}

\smallskip

\noindent
{\bf Legendre  polynomials of dimension $\pmb{n}$ and homogeneity $\pmb{d}$.} Recall the definition of the polynomials $L_{n,d}^{ \diamond}$
from Proposition~\ref{propleg}.  Again they may be formulated in 'Rodrigues form':
  For $n \ge 2$ and  $d\ge 2$ we have
   \begin{equation}\label{rodriguez}
     L_{n,d}^{ \diamond} (t)= (-1)^d \frac{\Gamma(\frac{n-1}{2})}{2^d \Gamma(d+\frac{n-1}{2})} \,\,
 (1-t^2)^{-\frac{n-3}{2}} \Big(\frac{\partial}{\partial t}\Big)^d(1-t^2)^{d+\frac{n-3}{2}}\,, \quad\ t\in [-1, 1]
   \end{equation}
    (see, e.g.,~\cite[Theorem~2.23]{atkinson2012spherical}). \noindent By inspecting  \eqref{rodjac},  \eqref{gegjac},\eqref{Gegenbauerpoly}, and \eqref{rodriguez}, we may relate the polynomials $L_{n,d}^{ \diamond}$ with  the Jacobi polynomials $P_d^{(\frac{n-3}{2}, \frac{n-3}{2})}$ as well as the
        Gegenbauer polynomials $C_d^{(\frac{n}{2}-1)}$.

\begin{lemma} \label{lemm: identity Gegenbauer} Given $n>2$ and $d \in \mathbb{N}$, for all $t \in [-1,1]$ 
\begin{align*}
L_{n,d}^{ \diamond}(t) & = \frac{d! \; (n-3)!}{(d+n-3)!} C^{(\frac{n}{2}-1)}_d(t)
\\
& = \frac{d! \; (n-3)!}{(d+n-3)!} \,
\frac{\sqrt{\pi} 2^{-\frac{n-3}{2}} \sqrt{(d+n-3)!}}{\sqrt{d! \; (d+\frac{n}{2} - 1) } \; \Gamma(\frac{n}{2}-1)}
\widehat{C}^{(\frac{n}{2}-1)}_d(t)
\\
& = \frac{d! \; \Gamma(\frac{n-1}{2})}{\Gamma(\frac{n-1}{2}+d)} \,
P_d^{(\frac{n-3}{2}, \frac{n-3}{2})}(t),
\end{align*}
here $\widehat{C}^{(\frac{n}{2}-1)}_{d} \colon [-1,1] \to \mathbb{R}$ stands for the normalized Gegenbauer polynomial, that is, its $L_2$-norm on the interval $[-1,1]$ with respect to the weight function $(1 - t^2)^{(\frac{n}{2}-1)–\frac{1}{2}} = (1 - t^2)^{\frac{n-3}{2}}$ is one.
\end{lemma}

  We emphasize some particular cases of importance that are known within the literature (see also \cite[Section~2.6]{atkinson2012spherical}): For $n=2$ we  get the $d$-th Chebyshev polynomial
 \begin{equation}\label{d=2}
     L_{2,d}^{ \diamond}(t)   = \cos(d \arccos t)\,,
   \end{equation}
 for $n=3$ the (standard) $d$-th Legendre polynomial
 \begin{equation}\label{d=3}
      L_{3,d}^{ \diamond}(t) =  \frac{1}{2^d d! } \,\,
  \Big(\frac{\partial}{\partial t}\Big)^d (t^2-1)^{d}\,,
   \end{equation}
and for $n=4$ the $d$-th  Chebyshev polynomial of second kind
\begin{equation}\label{d=4}
      L_{4,d}^{ \diamond}(t) =\frac{1}{d+1}\frac{\sin ((d+1) \arccos t)}{\sin (\arccos t)}\,.
   \end{equation}

\subsection{Harmonics}

We start with the calculation of the projection constant of $\mathcal{H}_{ d}(\mathbb{S}^{1})$,
so the dimension $n$ equals $2$.

\begin{proposition} \label{neu}
   For all $d \in \mathbb{N}$
\[
 \boldsymbol{\lambda}\big(\mathcal{H}_{ d}(\mathbb{S}^{1})\big) = \frac{4}{\pi}\,.
\]
\end{proposition}

\begin{proof}
We deduce from  Theorem~\ref{maino} that
\begin{align*}
\boldsymbol{\lambda}\big(\mathcal{H}_{ d}(\mathbb{S}^{1})\big)
&
= 2 \int_{\mathbb{S}^{1}} |L_{2,d}^{ \diamond}  (\langle e_1, \eta \rangle)| d\sigma_2(\eta)
\\&
= \frac{2}{2\pi}  \int_{\mathbb{S}^{1}} |L_{2,d}^{ \diamond}  (\langle e_1, \eta \rangle)|\, ds_2(\eta)
\\&
= \frac{1}{\pi}  \int_{-\pi}^{\pi} |L_{2,d}^{ \diamond}  (\cos t)| dt
=\frac{2}{\pi}  \int_{0}^{\pi} |L_{2,d}^{ \diamond}  (\cos t)| dt\,.
\end{align*}
Then by \eqref{d=2}
\begin{align*}
\int_{0}^{\pi} |L_{2,d}^{ \diamond}(t)  (\cos t)| dt
      &
      =
      \int_{0}^{\pi} \big|\cos \big(d \arccos (\cos t)\big)\big| dt
      \\&
      =\int_{0}^{\pi} |\cos (d t )| dt
      =
       \frac{1}{d} \int_{0}^{d\pi} |\cos t| dt
        =  \int_{0}^{\pi} |\cos t| dt   =  2\,.
    \end{align*}
           So, we get as desired \,
           $
       \boldsymbol{\lambda}\big(\mathcal{H}_{ d}(\mathbb{S}^{1})\big)
     =
     \frac{4}{\pi}$\,.
  \end{proof}

The previous equality is not surprising, since $H_d(\mathbb{S}^1)$ is isometrically isomorphic to
$\ell_1^2(\mathbb{C})$, and the result therefore follows directly from Gr{\"u}nbaum’s classical
integral formula for the projection constant \cite{grunbaum}:
\[
  \bold{\lambda}(\ell_1^2(\mathbb{C})) =\int_{\mathbb{T}^2} \abs{z_1 + z_2}\,dz_1dz_2,
\]
where $\mathbb{T} = \{z\in\mathbb{C} : \abs{z}=1\}$. Indeed, every function
$f \in H_d(\mathbb{S}^1)$ admits a representation
\[
f(\cos(\theta), \sin(\theta)) = a e^{id\theta} + b e^{-id\theta}, \qquad a,b \in \mathbb{C} \mbox{ and } \theta \in [0, 2 \pi).
\]
Thus the map
\[
T : \mathbb{C}^2 \longrightarrow H_d(\mathbb{S}^1), \qquad T(a,b)(\cos(\theta), \sin(\theta))
    = a e^{id\theta} + b e^{-id\theta},
\]
is a complex-linear bijection. For $f = T(a,b)$ and $z = e^{i\theta}$,
\[
f(\cos(\theta), \sin(\theta)) = a z^d + b z^{-d}
          = z^{-d}(a z^{2d} + b),
\]
and since $|z^{-d}| = 1$,
\[
|f(\cos(\theta), \sin(\theta))| = |a z^{2d} + b|.
\]
As $\theta$ varies, the point $w = z^{2d}$ runs over the unit circle, so
\[
\|f\|_\infty = \sup_{|w| = 1} |aw + b|.
\]
Consequently, $\|T(a,b)\|_\infty = |a| + |b|
                  = \|(a,b)\|_{\ell_1^2(\mathbb{C})}$.

\smallskip
Let us turn to the case $n>2$.
\smallskip


  

\begin{theorem} \label{projGegenbauerA}
Let $n\ge 3$ and $d \in \mathbb 
N$, and set 
$\alpha=\frac{n-3}{2}$.
Then the projection constant of the space $H_d(S^{n-1})$ of spherical
harmonics of degree $d$ satisfies
\[
\boldsymbol{\lambda}\big(\mathcal{H}_{d}(\mathbb{S}^{n-1})\big)
=
\frac{(2d+n-2)}{2^{n-2}}\,\frac{\Gamma(d+n-2)}{\Gamma(\frac{n-1}2)\,\Gamma\!\left(d+\frac{n-1}{2}\right)}
\int_{-1}^{1}
\bigl| P_d^{(\alpha,\alpha)}(t)\bigr|\,(1-t^2)^{\alpha}\,dt.
\]

 Alternatively, we have in terms of Gegenbauer polynomials that
  \[
    \boldsymbol{\lambda}\big(\mathcal{H}_{d}(\mathbb{S}^{n-1})\big)  = 
    \frac{2d+n-2}{\,n-2\,}\; 
\int_{\mathbb S^{n-1}}
\bigl|C_d^{(\frac{n-2}{2})}(\eta_1)\bigr|\,d\sigma_n(\eta).
   \]
\end{theorem}
\begin{proof}
    From Theorem~\ref{maino} and Lemma~\ref{lemm: identity Gegenbauer} we have
    \begin{align}
\label{eq: eq1 proof proj constant H_d}\boldsymbol{\lambda}\big(\mathcal{H}_{d}(\mathbb{S}^{n-1})\big) & =
N_{n,d}
\frac{d!\,\Gamma(\tfrac{n-1}{2})}
{\Gamma(d+\tfrac{n-1}{2})}
\int_{\mathbb S^{n-1}}
\big|P_d^{(\frac{n-3}{2},\frac{n-3}{2})}(\eta_1)\big|\,
d\sigma_n(\eta) \\
\nonumber &=N_{n,d}
\frac{d!\,\Gamma(\tfrac{n-1}{2})}
{\Gamma(d+\tfrac{n-1}{2})}
\frac{\Gamma(\tfrac n2)}{\sqrt{\pi}\,\Gamma(\tfrac{n-1}{2})}
\int_{-1}^{1} |P_d^{(\frac{n-3}{2},\frac{n-3}{2})}(t)|\,(1-t^2)^\frac{n-3}{2} dt,        
    \end{align}
where the last equality follows from the spherical reduction identity \eqref{umrechnen}.
Cancelling the factor \(\Gamma(\tfrac{n-1}{2})\) and  substituting \(N_{n,d}\) from \eqref{dimfor} yields
\begin{align*}
\boldsymbol{\lambda}\big(\mathcal{H}_{d}(\mathbb{S}^{n-1})\big) &=
\frac{(2d+n-2)\,(d+n-3)!}{d!\,(n-2)!}
\frac{d!\,\Gamma(\tfrac n2)}
{\sqrt{\pi}\,\Gamma(d+\tfrac{n-1}{2})}
\int_{-1}^{1}
\big|P_d^{(\frac{n-3}{2},\frac{n-3}{2})}(t)\big|(1-t^2)^\frac{n-3}{2} dt\\
&=
\frac{(2d+n-2)\,\Gamma(d+n-2)\,\Gamma(\tfrac n2)}
{\Gamma(n-1)\sqrt\pi\,\Gamma(d+\tfrac{n-1}{2})}
\int_{-1}^{1}
\big|P_d^{(\frac{n-3}{2},\frac{n-3}{2})}(t)\big|\,(1-t^2)^{\frac{n-3}{2}} dt.    
\end{align*}
An application of the Legendre duplication formula 
\begin{equation} \label{Legendre duplication formula}
    \Gamma(x)\, \Gamma\Big(x+ \frac{1}{2}\Big) = 2^{1-2x}\, \sqrt{\pi}\Gamma(2x)\,
\end{equation} 
{ for $x = \frac{n-1}{2}$,}
proves our first claim. The second statement follows from \eqref{eq: eq1 proof proj constant H_d} together with the relation between Gegenbauer and Jacobi polynomials from Lemma \ref{lemm: identity Gegenbauer}.
\end{proof}

\smallskip

For fixed dimension $n$, we now show the asymptotic order of $\boldsymbol{\lambda}\big(\mathcal{H}_{d}(\mathbb{S}^{n-1})\big)$ as the homogeneity degree $d$ goes to infinity. As the following theorem shows, this projection constant behaves as the square root of the dimension of $\mathcal{H}_{d}(\mathbb{S}^{n-1})$, meeting the Kadet-Snobar estimate. Recall that after fixing $n$ we for large $d$ have
\[
\dim \mathcal{H}_{d}(\mathbb{S}^{n-1}) \,\sim_{c(n)} \, d^{n-2}\,.
\]
More precisely, 
\begin{equation} \label{asymp dimension ratio}
\lim_{d \to \infty} \frac{\dim \mathcal{H}_{d}(\mathbb{S}^{n-1}) }{d^{n-2}} = \frac{2}{(n-2)!}\,.
\end{equation}
Indeed, the previous limit can be deduced from the fact that by Equation~\eqref{dimfor} the dimension is exactly
$$
\dim \mathcal{H}_{d}(\mathbb{S}^{n-1}) = \frac{(2d+n-2) \Gamma(d+n-2)}{\Gamma(d+1) (n-2)!}\,,
$$ 
combined with the following well-known formula for the asymptotic of ratios of Gamma functions: Given fixed numbers $a,b >0$,
\begin{align}\label{mainasym}
\lim_{x\to \infty} \frac{\Gamma(x + a)}{ \Gamma(x + b) x^{a-b}} =1\,. 
\end{align}

\begin{theorem} \label{thm: asymptotic order real homogeneous harmonics} For $n>2$
   \begin{equation} \label{Eq: limit projection harmonic A}
   \lim_{d \to \infty} \frac{\boldsymbol{\lambda}\big(\mathcal{H}_{d}(\mathbb{S}^{n-1})\big)}{\sqrt{\dim \mathcal{H}_{d}(\mathbb{S}^{n-1})}} \,\, = \,\,\frac{  2^{n-\frac{1}{2}}  \; \; }{ \sqrt{(n-2)!} \;\; } 
\frac{\Gamma\big(\frac{n}{4}\big)^2}{\pi^{2}}.
   \end{equation}
    In particular,
   \begin{equation} \label{Eq: limit projection harmonic B}
   \lim_{d \to \infty} \frac{\boldsymbol{\lambda}\big(\mathcal{H}_{d}(\mathbb{S}^{n-1})\big)}{d^{\frac{n-2}{2}}} \,\, = \,\,\frac{  2^{n}  \; \; }{ {(n-2)!} \;\; } 
\frac{\Gamma\big(\frac{n}{4}\big)^2}{\pi^{2}}.
   \end{equation}
\end{theorem}

\begin{proof} By Theorem \ref{projGegenbauerA} and Lemma \ref{lemm: identity Gegenbauer} we have from the formula \eqref{dimfor}
\begin{align*}
&
\frac{\boldsymbol{\lambda}\big(\mathcal{H}_{d}(\mathbb{S}^{n-1})\big)}{\sqrt{\mbox{dim}\big(\mathcal{H}_{d}(\mathbb{S}^{n-1})\big)}} 
\\
&
= \sqrt{\frac{d!(n-2)!}{(2d+n-2)(d+n-3)!}} \,\frac{2d+n-2}{(n-2)}  \;\; \frac{\sqrt{\pi} 2^{-\frac{n-3}{2}} \sqrt{(d+n-3)!}}{\sqrt{d! } \; \sqrt{d+\frac{n}{2}-1} \, \Gamma(\frac{n}{2}-1) } \int_{\mathbb{S}^{n-1}} \big | \widehat{C}^{(\frac{n}{2}-1)}_d(\eta_1) \big|\; d\sigma_n(\eta) 
\\
& = \frac{\sqrt{\pi}  \;\;2^{-\frac{n-4}{2}} \;\; (n-3)!}{ \sqrt{(n-2)!} \;\; \Gamma(\frac{n}{2}-1)} \int_{\mathbb{S}^{n-1}}
\big|\widehat{C}^{(\frac{n}{2}-1)}_d(\eta_1)\big|\,d\sigma_n(\eta)\,,
\end{align*}
and  consequently it follows from Equation~\eqref{umrechnen} that
\begin{align*}
\frac{\boldsymbol{\lambda}\big(\mathcal{H}_{d}(\mathbb{S}^{n-1})\big)}{\sqrt{\mbox{dim}\big(\mathcal{H}_{d}(\mathbb{S}^{n-1})\big)}} & =  \frac{  2^{-\frac{n-4}{2}} \;\: (n-3)! \; \; \Gamma(\frac{n}{2})}{ \sqrt{(n-2)!} \;\; \Gamma(\frac{n}{2}-1) \; \; \Gamma(\frac{n-1}{2})} \int_{-1}^1
\big|\widehat{C}^{(\frac{n}{2}-1)}_d(t)\big| \; (1-t^2)^{\frac{n-3}{2}} dt\,.
\end{align*}
From  statement 2) of \cite[Theorem 1]{aptekarev1995asymptotic} we know that
\[
\lim_{d\to \infty} \; \; \int_{-1}^1
|\hat{C}^{(\frac{n}{2}-1)}_d(t)| \; (1-t^2)^{\frac{n-3}{2}} dt
=\left(\frac{2}{\pi}\right)^{3/2} \int_0^\pi (1-\cos^2 \theta)^{\frac{n-2}{4}} \;d\theta \,
= \frac{2^{\frac{n+1}{2}}}{\pi^{3/2}}\int_0^\pi \sin^{\frac{n-2}{2}}\Big(\frac{\theta}{2}\Big)\cos^{\frac{n-2}{2}} \Big(\frac{\theta}{2}\Big) \;d\theta
.\]
The last integral is  the Beta function evaluated at $\big(\frac{n}{4},\frac{n}{4}\big)$, thus
\[
\lim_{d\to \infty} \; \; \int_{-1}^1
|\hat{C}^{(\frac{n}{2}-1)}_d(t)| \; (1-t^2)^{\frac{n-3}{2}} dt
= \frac{2^\frac{n+1}{2}}{\pi^{3/2}}\frac{\Gamma\big(\frac{n}{4}\big)^2}{\Gamma\big(\frac{n}{2}\big)}.
\]

This shows that
\begin{align*}
\lim_{d \to \infty} \frac{\boldsymbol{\lambda}\big(\mathcal{H}_{d}(\mathbb{S}^{n-1})\big)}{\sqrt{\mbox{dim}\big(\mathcal{H}_{d}(\mathbb{S}^{n-1})\big)}}  & = \frac{  2^{-\frac{n-4}{2}} \;\: (n-3)! \; \; \Gamma(\frac{n}{2})}{ \sqrt{(n-2)!} \;\; \Gamma(\frac{n}{2}-1) \; \; \Gamma(\frac{n-1}{2})} \;\; 
\frac{2^\frac{n+1}{2}}{\pi^{3/2}}\frac{\Gamma\big(\frac{n}{4}\big)^2}{\Gamma\big(\frac{n}{2}\big)} \\
& = \frac{  2^{\frac{5}{2}} \;\: (n-3)! \; \; }{ \sqrt{(n-2)!} \;\; \Gamma(\frac{n}{2}-1) \; \; \Gamma(\frac{n-1}{2})} \;\; 
\frac{\Gamma\big(\frac{n}{4}\big)^2}{\pi^{3/2}}\\
& = \frac{  2^{n-\frac{1}{2}}  \; \; }{ \sqrt{(n-2)!} \;\; } 
\frac{\Gamma\big(\frac{n}{4}\big)^2}{\pi^{2}},
\end{align*}
where for the last equality we applied the Legendre duplication formula \eqref{Legendre duplication formula}
with $x=\frac{n}{2}-1,$
which is exactly \eqref{Eq: limit projection harmonic A}. On the other hand, \eqref{Eq: limit projection harmonic B} follows from the previous limit and Equation~\eqref{asymp dimension ratio}.
\end{proof}

\subsection{Homogeneous polynomials}

As before in Proposition~\ref{neu} we start with the $2$-dimensional case.

\begin{proposition} \label{prop : 2 dimensional homogeneous}
For all $d \in \mathbb{N}$
\[
 \boldsymbol{\lambda}\big(\mathcal{P}_{ d}(\mathbb{S}^{1})\big) =
\frac{1}{2\pi} \int_{0}^{2\pi}\bigg| \frac{\sin\big(\frac{d+1}{2}t \big)}{\sin\big(\frac{t}{2} \big)}\bigg|\,dt\,.
\]
In particular,
\[
\lim_{d \to \infty}  \frac{\boldsymbol{\lambda}\big(\mathcal{P}_{ d}(\mathbb{S}^{1})\big)}{\log d} \,=\, \frac{4}{\pi^2}\,.
\]
\end{proposition}

It should be pointed out that in the case of the complex Hilbert space $\ell_2^2(\mathbb{C})$, the projection constant $\boldsymbol{\lambda}\big(\mathcal{P}_{d}(\ell_2^2(\mathbb{C})\big)$ is bounded by $2$ for any degree $d$ (see again~\eqref{surprise} and also~\eqref{surpriseB}). However, the previous result shows that for the real case we have a~significant difference. In fact, the projection constant of $\mathcal{P}_{d}(\ell_2^2(\mathbb{R}))=\mathcal{P}_{d}(\mathbb{S}^{1})$ has a~logarithmic increase with respect to $d$.

\begin{proof}
We follow the ideas of the proof of  Propositions~\ref{neu}.
    Suppose first that $d$ is even.
    Note that $N_{2,0}=1$ and that, on the other hand, by Equation \eqref{dimfor} we have $N_{2,\ell}=2$ for every $\ell\geq 1$.
Then we conclude from Theorem~\ref{maino} that
\begin{align*}
    \boldsymbol{\lambda}\big(\mathcal{P}_{ d}(\mathbb{S}^{1})\big)
    &
    =  \int_{\mathbb{S}^{1}} \big|1+ 2 \sum_{\ell=1}^{\frac{d}{2}} L_{2,2\ell}^{\diamond}(\langle e_1, \eta \rangle)\big| d\sigma_2(\eta)
    \\&
    =\frac{1}{2\pi}\int_{\mathbb{S}^{1}} \big|1+ 2 \sum_{\ell=1}^{\frac{d}{2}} L_{2,2\ell}^{\diamond}(\langle e_1, \eta \rangle)\big| d\lambda_2(\eta)
    \\&
        =\frac{1}{2\pi}\int_{0}^{2\pi} \big|1+ 2 \sum_{\ell=1}^{\frac{d}{2}} L_{2,2\ell}^{\diamond}(\cos t)\big| dt
                \\&
        =\frac{1}{2\pi}\Big(\int_{0}^{\pi} \big|1+ 2 \sum_{\ell=1}^{\frac{d}{2}} L_{2,2\ell}^{\diamond}(\cos t)\big| dt
        +
        \int_{\pi}^{2\pi} \big|1+ 2 \sum_{\ell=1}^{\frac{d}{2}} L_{2,2\ell}^{\diamond}(\cos t)\big| dt\Big)\,.
            \end{align*}
By \eqref{d=2} we know that $L_{2,2\ell}^{ \diamond}$ is the  $d$-th Chebyshev polynomial, that is,
 \begin{equation*}
     L_{2,2\ell}^{ \diamond}(t)   = \cos(2\ell \arccos t)\,,
   \end{equation*}
and so we may rewrite the preceding two integrals as follows:
\begin{align*}
\int_{0}^{\pi} \big|1+ 2 \sum_{\ell=1}^{\frac{d}{2}} L_{2,2\ell}^{\diamond}(\cos t)\big|\,dt
&=
  \int_{0}^{\pi} \Big|1+ 2 \sum_{\ell=1}^{\frac{d}{2}} \cos (2\ell t)\Big|\,dt
  \\&
= \frac{1}{2}  \int_{0}^{2\pi} \Big|1+ 2 \sum_{\ell=1}^{\frac{d}{2}} \cos (\ell s)\Big|\,ds =  \frac{1}{2} \int_{0}^{2 \pi}  \bigg|\frac{\sin\big(( \frac{d+1}{2})s \big)}{\sin\big(\frac{s}{2} \big)}\bigg| \,ds
\end{align*}
and
\begin{align*}
\int_{\pi}^{2\pi} \Big|1+ 2 \sum_{\ell=1}^{\frac{d}{2}} L_{2,2\ell}^{\diamond}(\cos t)\Big|\,dt
&
=
\int_{\pi}^{2\pi} \Big|1+ 2 \sum_{\ell=1}^{\frac{d}{2}} \cos \big(2\ell (2\pi - t)\big)\Big|\,dt
\\
& =
  \frac{1}{2} \int_{0}^{2\pi} \Big|1+ 2 \sum_{\ell=1}^{\frac{d}{2}} \cos (\ell s) \Big|\,ds
= \frac{1}{2} \int_{0}^{2\pi}  \bigg|\frac{\sin\big(( \frac{d+1}{2})s \big)}{\sin\big(\frac{s}{2} \big)}\bigg|\,ds
\,.
\end{align*}
Combining,  we obtain the required formula for the even case.

Suppose now that $d$ is odd. In this case  Theorem~\ref{maino} gives that
\begin{align*}
    \boldsymbol{\lambda}\big(\mathcal{P}_{ d}(\mathbb{S}^{1})\big)
    &
    =  \int_{\mathbb{S}^{1}} \big| 2 \sum_{\ell=0}^{\floor{\frac{d}{2}}} L_{2,2\ell+1}^{\diamond}(\langle e_1, \eta \rangle)\big| d\sigma_2(\eta)
    \\&
        =\frac{1}{2\pi}\Big(\int_{0}^{\pi} \big| 2 \sum_{\ell=0}^{\floor{\frac{d}{2}}} L_{2,2\ell+1}^{\diamond}(\cos t)\big| dt
        +
        \int_{\pi}^{2\pi} \big| 2 \sum_{\ell=0}^{\floor{\frac{d}{2}}} L_{2,2\ell+1}^{\diamond}(\cos t)\big| dt\Big)\,.
    \\&
    =\frac{1}{2\pi}\Big(\int_{0}^{\pi} \big| 2 \sum_{\ell=0}^{\floor{\frac{d}{2}}} \cos((2\ell +1) t\big| dt
        +
        \int_{\pi}^{2\pi} \big| 2 \sum_{\ell=0}^{\floor{\frac{d}{2}}} \cos((2\ell +1)(2\pi- t)\big| dt\Big)\,.
\end{align*}
Now note that
$$ \sum_{\ell=0}^{\floor{\frac{d}{2}}} \cos\Big((2\ell+1)\frac{s}{2}\Big) = \textbf{Re}\, \, \Bigg( \sum_{\ell=0}^{\floor{\frac{d}{2}}} e^{i (2 \ell +1) \frac{s}{2}} \Bigg) = \frac{1}{2} \frac{\sin\big((\floor{\frac{d}{2}}+1\big)s)}{\sin(\frac{s}{2})}= \frac{1}{2} \frac{\sin\big((\frac{d+1}{2})s\big)}{\sin(\frac{s}{2})}\,,$$
and then the  proof follows the same steps as before, with the substitution $t=\frac{s}{2}$ as previously employed.

Finally, the limiting formula follows from the well-known properties of the Dirichlet kernel, see \cite[Chapter II, eq.(12.1)]{zygmund2002trigonometric}.
\end{proof}

We go on with the case $n>2$. Recall that the floor function $\lfloor x \rfloor$ denotes the largest integer less than or equal to $x$.
The following result is an analog of Theorem~\ref{projGegenbauerA}.

\begin{theorem} \label{projGegenbauerB}
    Given $n>2$ and $d\in \mathbb{N}$, 
\begin{align*}
    \boldsymbol{\lambda}\big(\mathcal{P}_{d}(\mathbb{S}^{n-1})\big) \,\, = \,\,
    \frac{1}{2\sqrt{\pi}} \,\frac{\Gamma(\frac{n}{2})\Gamma(d+{n})}{\Gamma(n-1)\Gamma(d+\frac{n+1}{2})}
     \, \int_{-1}^{1} \big| P_d^{(\frac{n-1}{2},\frac{n-1}{2})}(t)\big| (1-t^2)^{\frac{n-3}{2}} dt\,.
    \end{align*}
     \end{theorem}

When $d=1$, the preceding outcome aligns with Rutovitz's findings for the real Hilbert space from Equation~\eqref{beginning}. To see this, note   that then  the integral in the preceding formula equals $\frac{n+1}{n-1}$, and use \eqref{Legendre duplication formula} twice. To delve deeper into this interpretation, see  the concluding remark in Section \ref{final remark}.

\begin{proof}
We have to distinguish the even degrees $d$ form the   odd ones, and start with the even case.
By Theorem~\ref{maino} one has
\begin{equation*}
 \boldsymbol{\lambda}\big(\mathcal{P}_{d}(\mathbb{S}^{n-1})\big)
 =
\int_{\mathbb{S}^{n-1}}\,
\big| \mathbf{k}_{\mathcal{P}_{d}(\mathbb{S}^{n-1})}(e_1, \eta )\big| \,d\sigma_n(\eta)
\end{equation*}
with
$$\mathbf{k}_{\mathcal{P}_{d}(\mathbb{S}^{n-1})}(e_1, \eta )
 =  \sum_{j=0}^{\frac{d}{2}} N_{n,j} \,\,L_{n,2j}^{\diamond}(\eta_1)\,.
$$
Define
$
\alpha = \frac{n-3}{2}\,.
$
Then by Lemma~\ref{lemm: identity Gegenbauer}  and Equation~\eqref{dimfor} we get  that
\begin{align*}
     \mathbf{k}_{\mathcal{P}_{d}(\mathbb{S}^{n-1})}(e_1, \eta )
    =
  \sum_{j=0}^{\frac{d}{2}} \frac{(4j+n-2)(2j+n-3)!}{(2j)!(n-2)!}
  \,\,
    \frac{(2j)! \;\; (n-3)!}{(2j+n-3)!}
    \,\,
\frac{\Gamma(n-2 +2j)\Gamma(\frac{n-1}{2})}{\Gamma(n-2)\Gamma(\frac{n-1}{2}+2j)}
\,\,
P_{2j}^{(\alpha, \alpha)}(\eta_1)
\,,
\end{align*}
and hence
\begin{align} \label{kern}
     \mathbf{k}_{\mathcal{P}_{d}(\mathbb{S}^{n-1})}(e_1, \eta )
   =
\frac{\Gamma(\frac{n-1}{2})}{\Gamma(n-1)}\,\, \sum_{j=0}^{\frac{d}{2}}\, \frac{2(2j +\alpha + \frac{1}{2})\Gamma(2(j+\alpha +\frac{1}{2}))}{\Gamma(2j+\alpha+1)}\,\,
P_{2j}^{(\alpha, \alpha)}(\eta_1)\,.
\end{align}
By \cite[Theorem 4.1, p.~59 and (4.5.3), p.~71]{szeg1939orthogonal}, we have
\begin{align}\label{szegoe-even}
  P_{2j}^{(\alpha, \alpha)}(\eta_1) = \frac{\Gamma(2j+\alpha+1)j!}{\Gamma(j+\alpha+1)(2j)!} \,\,
P_{j}^{(\alpha, -\frac{1}{2})}(2\eta_1^2-1)\,,
\end{align}
and
   \begin{align}\label{szegoe-evenB}
   \sum_{j=0}^{\frac{d}{2}}\,
\frac{(2j + \alpha -\frac{1}{2}+1)\Gamma(j+\alpha-\frac{1}{2}+1)}{ \Gamma(j-\frac{1}{2}+1)}
\,
P_{j}^{(\alpha, -\frac{1}{2})}(2\eta_1^2-1)
=
 \frac{\Gamma(\frac{d}{2} + \alpha -\frac{1}{2} +2)}{\Gamma(\frac{d}{2}  -\frac{1}{2} +1)}
P_{\frac{d}{2}}^{(\alpha+1, -\frac{1}{2})}(2\eta_1^2-1)\,.
        \end{align}
Then, applying \eqref{szegoe-even} and \eqref{szegoe-evenB} to \eqref{kern}, we get
\begin{align} \label{rearranging-even}
\begin{split}
    \mathbf{k}_{\mathcal{P}_{d}(\mathbb{S}^{n-1})}(e_1, \eta )
    =
\frac{\Gamma(\frac{n-1}{2})}{\Gamma(n-1)}
\,
\sum_{j=0}^{\frac{d}{2}}\,
C(n,j)
\,
\frac{(2j + \alpha -\frac{1}{2}+1)\Gamma(j+\alpha-\frac{1}{2}+1)}{ \Gamma(j-\frac{1}{2}+1)}
\,
P_{j}^{(\alpha, -\frac{1}{2})}(2\eta_1^2-1)\,,
\end{split}
 \end{align}
 where
 \begin{align*}
   C(n,j)
   &
   =
\frac{2(2j +\alpha + \frac{1}{2})\Gamma(2(j+\alpha +\frac{1}{2}))}{\Gamma(2j+\alpha+1)}
\,\,
\frac{\Gamma(2j+\alpha+1)j!}{\Gamma(j+\alpha+1)(2j)!}
\,\,
\frac{ \Gamma(j-\frac{1}{2}+1)}{(2j + \alpha -\frac{1}{2}+1)\Gamma(j+\alpha-\frac{1}{2}+1)}
\\&
=
2
\,\,
\frac{\Gamma(2(j+\alpha +\frac{1}{2}))j!}{\Gamma((j+\alpha+\frac{1}{2})+\frac{1}{2})(2j)!}
\,\,
\frac{ \Gamma(j+\frac{1}{2})}{\Gamma(j+\alpha+\frac{1}{2})}\,.
 \end{align*}
We claim that
$
C(n,j) = 2^{n-2}\,.
$
Indeed, we use the formula
$$\Gamma\Big(j+\frac{1}{2}\Big) = \frac{(2j)!\sqrt{\pi}}{4^j j!}$$
 and for $x = j + \alpha + \frac{1}{2}$ the
Legendre duplication formula
\eqref{Legendre duplication formula},
to see that 
\[
C(n,j) = 2 \frac{\Gamma(2x)}{\Gamma(x + \frac{1}{2})\Gamma(x)} \frac{j!}{(2j)!}
\Gamma\Big(j+\frac{1}{2}\Big) = 2^{n-2}.
\]
            Summarizing, we get
 \begin{align*}
 \mathbf{k}_{\mathcal{P}_{d}(\mathbb{S}^{n-1})}(e_1, \eta )
    = 2^{n-2}
\frac{\Gamma(\frac{n-1}{2})}{\Gamma(n-1)}
\,
\frac{\Gamma(\frac{d}{2}+\frac{n}{2})}{\Gamma(\frac{d}{2}+\frac{1}{2})}
\,
P_{\frac{d}{2}}^{(\alpha+1, -\frac{1}{2})}(2\eta_1^2-1)\,.
           \end{align*}
           Finally, we once again use \cite[Theorem 4.1, p.59]{szeg1939orthogonal} to obtain
           \begin{align*}
             P_{\frac{d}{2}}^{(\alpha+1, -\frac{1}{2})}(2\eta_1^2-1)
             =
             \frac{\Gamma(\frac{d}{2} + \frac{n-1}{2} +1) \Gamma(d+1)}{\Gamma(d + \frac{n-1}{2} +1)\Gamma(\frac{d}{2}+1)}
             P_{d}^{(\frac{n-1}{2},\frac{n-1}{2})}(\eta_1)\,,
           \end{align*}
and conclude that
 \begin{align*}
 \mathbf{k}_{\mathcal{P}_{d}(\mathbb{S}^{n-1})}(e_1, \eta )
    &=
    2^{n-2}
\frac{\Gamma(\frac{n-1}{2})}{\Gamma(n-1)}
\,
\frac{\Gamma(\frac{d+n}{2})\,\Gamma(\frac{d+n}{2} + \frac{1}{2})\, \Gamma(d+1)}{\Gamma(\frac{d}{2}+\frac{1}{2})\,\Gamma(\frac{d}{2}+1)\,\Gamma(d + \frac{n+1}{2} )}\,\,
             P_{d}^{(\frac{n-1}{2},\frac{n-1}{2})}(\eta_1)\,\\
    &= \frac12 \, \frac{\Gamma(\frac{n-1}{2})}{\Gamma(n-1)} \,  \frac{ \Gamma(d+n)}{\Gamma(d + \frac{n+1}{2} )}\,\,
             P_{d}^{(\frac{n-1}{2},\frac{n-1}{2})}(\eta_1)\,,
           \end{align*}
           where, for the last equality, we have used \eqref{Legendre duplication formula} twice.

Consequently, by Equation~\eqref{umrechnen} we have
\begin{align} \label{fin}
\begin{split}
   &
 \boldsymbol{\lambda}\big(\mathcal{P}_{d}(\mathbb{S}^{n-1})\big)
    =\frac12 \, \frac{\Gamma(\frac{n-1}{2})}{\Gamma(n-1)} \,  \frac{ \Gamma(d+n)}{\Gamma(d + \frac{n+1}{2} )}\,
           \int_{\mathbb{S}^{n-1}}\, \big| P_{d}^{(\frac{n-1}{2},\frac{n-1}{2})}(\eta_1)\big|  \,d\sigma_n(\eta)
        \\[1ex]&
                    =
\frac12 \, \frac{\Gamma(\frac{n-1}{2})}{\Gamma(n-1)} \,  \frac{ \Gamma(d+n)}{\Gamma(d + \frac{n+1}{2} )}\,\, \frac{\Gamma(\frac{n}{2})}{\sqrt{\pi}\, \Gamma(\frac{n-1}{2})}\int_{-1}^1\, \big| P_{d}^{(\frac{n-1}{2},\frac{n-1}{2})}(t)\big|  (1-t^2)^{\frac{n-3}{2}} dt\,,
\end{split}
       \end{align}
and this finishes  the even case of our claim.

The proof of the  odd case is similar - but with some subtle differences.
If $d$ is odd, then
we are going to replace \eqref{szegoe-even}  by
\begin{align}\label{szegoe-odd}
  P_{2j+1}^{(\alpha, \alpha)}(\eta_1) = \frac{\Gamma(2j+\alpha+2)j!}{\Gamma(j+\alpha+1)(2j+1)!} \,\,
\eta_1 P_{j}^{(\alpha, \frac{1}{2})}(2\eta_1^2-1)\,,
\end{align}
and \eqref{szegoe-evenB} by
\begin{align}\label{szegoe-oddB}
 \sum_{j=0}^{\lfloor\frac{d}{2}\rfloor}\,
\,
\frac{(2j + \alpha +\frac{1}{2}+1)\Gamma(j+\alpha+\frac{1}{2}+1)}{\Gamma(j+\frac{1}{2}+1)}
\,
\eta_1 P_{j}^{(\alpha, \frac{1}{2})}(2\eta_1^2-1)
=
\frac{\Gamma(\lfloor\frac{d}{2}\rfloor + \alpha +\frac{1}{2} +2)}{\Gamma(\lfloor\frac{d}{2}\rfloor  +\frac{1}{2} +1)}
\,
\eta_1 P_{\lfloor\frac{d}{2}\rfloor}^{(\alpha, \frac{1}{2})}(2\eta_1^2-1)\,,
\end{align}
where again $\alpha = \frac{n-3}{2}$ (see  \cite[Theorem 4.1, p.59 and (4.5.3), p.~71]{szeg1939orthogonal}).
Starting as above,
we have
\begin{align}\label{kero-odd}
  \mathbf{k}_{\mathcal{P}_{d}(\mathbb{S}^{n-1})}(e_1, \eta )
= \sum_{j=0}^{\lfloor\frac{d}{2}\rfloor} N_{n,2j+1} \,\,L_{n,2j+1}^{ \diamond}(\eta_1)\,,
\end{align}
and then  (as for \eqref{kern})  we obtain
\begin{align} \label{kern-odd}
     \mathbf{k}_{\mathcal{P}_{d}(\mathbb{S}^{n-1})}(e_1, \eta )
   =
\frac{\Gamma(\frac{n-1}{2})}{\Gamma(n-1)}\,\, \sum_{j=0}^{\lfloor\frac{d}{2}\rfloor}\,
\frac{2((2j+1) +\alpha + \frac{1}{2})
   \Gamma(2(j+\alpha +1)) }{\Gamma((2j+1)+\alpha +1)}\,\,
P_{2j+1}^{(\alpha, \alpha)}(\eta_1)\,.
\end{align}
 Then as in \eqref{rearranging-even}
\begin{align} \label{rearranging-odd}
    \mathbf{k}_{\mathcal{P}_{d}(\mathbb{S}^{n-1})}(e_1, \eta )
    =
\frac{\Gamma(\frac{n-1}{2}))}{\Gamma(n-1)}
\,
\sum_{j=0}^{\lfloor\frac{d}{2}\rfloor}\,
C(n,j)
\,
\frac{(2j + \alpha +\frac{1}{2}+1)\Gamma(j+\alpha+\frac{1}{2}+1)}{ \Gamma(j+\frac{1}{2}+1)}
\,
\eta_1 P_{j}^{(\alpha, \frac{1}{2})}(2\eta_1^2-1)\,,
 \end{align}
 where
\begin{align*} \label{C-odd}
\begin{split}
&
  C(n,j)
  \\&
       =
              \frac{2((2j+1) +\alpha + \frac{1}{2})
   \Gamma(2(j+\alpha +1)) }{\Gamma((2j+1)+\alpha +1)}\,
   \frac{\Gamma(2j+\alpha+2)j!}{\Gamma(j+\alpha+1)(2j+1)!} \,
   \frac{\Gamma(j+\frac{1}{2}+1)}{(2j + \alpha +\frac{1}{2}+1)\Gamma(j+\alpha+\frac{1}{2}+1)}
   \\&
       =2\,\,
                \frac{\Gamma(2(j+\alpha +1)) j!}{\Gamma(j+\alpha+1)(2j+1)!} \,
   \frac{\Gamma(j+\frac{1}{2}+1)}{\Gamma(j+\alpha+\frac{1}{2}+1)}
   \,.
 \end{split}
   \end{align*}
      Using that
   $$\Gamma((j+1)+\frac{1}{2}) = \frac{\sqrt{\pi}}{4^{j+1}} \frac{(2(j+1))!}{(j+1)!}= \frac{1}{2}\frac{\sqrt{\pi}}{4^j }\frac{(2j+1)!}{j!}$$
      and again the Legendre duplication formula  from
      \eqref{Legendre duplication formula}
      (now for $x=j+\alpha +1$), we as above get
      \[
      C(n,j) =  \frac{\Gamma(2x)}{\Gamma(x)\Gamma(x+\frac{1}{2})} \frac{\sqrt{\pi}}{4^j}
      =2^{n-2}\,.
      \]
Then we derive from \eqref{rearranging-odd} and \eqref{szegoe-oddB} that
\begin{align} \label{rearranging-oddBB}
    \mathbf{k}_{\mathcal{P}_{d}(\mathbb{S}^{n-1})}(e_1, \eta )
    =
\frac{\Gamma(\frac{n-1}{2})}{\Gamma(n-1)} 2^{n-2}
\,
\frac{\Gamma(\lfloor\frac{d}{2}\rfloor + \alpha +\frac{1}{2} +2)}{\Gamma(\lfloor\frac{d}{2}\rfloor  +\frac{1}{2} +1)}
\,
\eta_1 P_{\lfloor\frac{d}{2}\rfloor}^{(\alpha, \frac{1}{2})}(2\eta_1^2-1)\,,
 \end{align}
 Reversing the process, using  \cite[Theorem 4.1, p.59]{szeg1939orthogonal},  we conclude from
           \begin{align*}
             \eta_1P_{\lfloor\frac{d}{2}\rfloor}^{(\alpha+1, \frac{1}{2})}(2\eta_1^2-1)
             =
             \frac{\Gamma(\lfloor\frac{d}{2}\rfloor + \frac{n-1}{2} +1) \Gamma(2\lfloor\frac{d}{2}\rfloor+2)}
             {\Gamma(2\lfloor\frac{d}{2}\rfloor + \frac{n-1}{2} +2)\Gamma(\lfloor\frac{d}{2}\rfloor+1)}
             P_{2\lfloor\frac{d}{2}\rfloor+1}^{(\frac{n-1}{2},\frac{n-1}{2})}(\eta_1)
           \end{align*}
           that
           \begin{align*}
    \mathbf{k}_{\mathcal{P}_{d}(\mathbb{S}^{n-1})}(e_1, \eta )
    =
\frac{\Gamma(\frac{n-1}{2})}{\Gamma(n-1)} 2^{n-2}
\,
\frac{\Gamma(\lfloor\frac{d}{2}\rfloor + \frac{n}{2}+1)}{\Gamma(\lfloor\frac{d}{2}\rfloor  +\frac{1}{2} +1)}
\,
\frac{\Gamma(\lfloor\frac{d}{2}\rfloor + \frac{n}{2}+ \frac{1}{2} ) \Gamma(2\lfloor\frac{d}{2}\rfloor+2)}
             {\Gamma(2\lfloor\frac{d}{2}\rfloor + \frac{n}{2} +\frac{3}{2})\Gamma(\lfloor\frac{d}{2}\rfloor+1)}
             P_{2\lfloor\frac{d}{2}\rfloor+1}^{(\frac{n-1}{2},\frac{n-1}{2})}(\eta_1)\,.
 \end{align*}

But since $\lfloor\frac{d}{2}\rfloor = \frac{d-1}{2}$, we have, using again \eqref{Legendre duplication formula} twice,
\begin{align*}
             \mathbf{k}_{\mathcal{P}_{d}(\mathbb{S}^{n-1})}(e_1, \eta )
  & =\, \frac{\Gamma(\frac{n-1}{2})}{\Gamma(n-1)} 2^{n-2}
\,
\frac{\Gamma(\frac{d+n}{2}+\frac{1}{2})}{\Gamma(\frac{d+1}{2}+\frac{1}{2})} \,\,   \frac{ \Gamma(\frac{d+n}{2})}{\Gamma(\frac{d+1}{2})}
        \,\,
                \frac{ \Gamma(d+1)}{\Gamma(d+ \frac{n}{2} +\frac{1}{2})} P_{d}^{(\frac{n-1}{2},\frac{n-1}{2})}(\eta_1)\\
    &= \frac12 \, \frac{\Gamma(\frac{n-1}{2})}{\Gamma(n-1)} \,  \frac{ \Gamma(d+n)}{\Gamma(d + \frac{n+1}{2} )}\,\,
             P_{d}^{(\frac{n-1}{2},\frac{n-1}{2})}(\eta_1)\,.
 \end{align*}
 The proof finishes as in \eqref{fin}.
\end{proof}

\smallskip
The subsequent result, in particular, shows that here $\boldsymbol{\lambda}\big(\mathcal{P}_{d}(\mathbb{S}^{n-1})\big)$ exhibits a strictly smaller order than the one given by the Kadets-Snobar upper estimate, since, for fixed $n$, we have for large $d$
(see again the Equations~\eqref{dimformuA},
~\eqref{aprilo},and~\eqref{mainasym})
\[
\dim \mathcal{P}_{d}(\mathbb{S}^{n-1}) \,\sim_{c(n)} \, d^{n-1}\,.
\]

\begin{corollary} \label{repetition}
For each integer $n >2$
    \[
    \lim_{d \to \infty}
    \frac{\boldsymbol{\lambda}\big(\mathcal{P}_{ d}(\mathbb{S}^{n-1})\big)}{d^\frac{n-2}{2}}
    \,=  \,   \frac{2^{n+1} }{\pi^2\,({n}-{2})}  \,\frac{\Gamma(\frac{n}{4}+\frac{1}{2})^2}{\Gamma(n-1)}  
    \,.
  \]
  \end{corollary}

\begin{proof}
  The following limit  from \cite[§19, Equation (64-1), p. 84]{szeg1933Asymptotische} is crucial for our purposes: For $\alpha, \beta, \lambda, \mu > -1$
                                 such that $2 \lambda > \alpha - \frac{3}{2}$
   \begin{align} \label{doro-eq}
   \begin{split}
      \lim_{d \to \infty} \,\, \sqrt{d}\,\,     
      \int_{0}^{1}  (1-x)^\lambda (1+x)^\mu &|P^{\alpha, \beta}_d(x)| dx  \\&
                   \,= \, \frac{2^{\lambda+\mu +1}}{ \sqrt{\pi}} \frac{2}{\pi}
    \int_{0}^{\frac{\pi}{2}}  \big(\sin (\theta/2) \big)^{2\lambda-\alpha +\frac{1}{2}}
    \big(\cos (\theta/2) \big)^{2\mu-\beta +\frac{1}{2}} d \theta\,,
   \end{split}
        \end{align}
   and we are going to use this equation for $\alpha = \beta = \frac{n-1}{2}$
   and $\lambda = \mu = \frac{n-3}{2}$. Moreover, we need to know that by
       \cite[ 4.1.3, p.~59]{szeg1939orthogonal}
    Jacobi polynomials satisfy
    \begin{align*} \label{SZ1}
      P_d^{(\alpha, \beta)} (t) = (-1)^d P_d^{(\beta,\alpha)} (-t), \quad\, t\in [-1, 1]\,,
    \end{align*}
    hence, as $d\to\infty$,
\begin{align*}
    \sqrt{d}\int_{-1}^{1} \big| P_d^{(\frac{n-1}{2},\frac{n-1}{2})}(t)\big| (1-t^2)^{\frac{n-3}{2}} dt
    & = \,   2\sqrt{d}
    \int_{0}^{1} \big| P_d^{(\frac{n-1}{2},\frac{n-1}{2})}(t)\big| (1-t^2)^{\frac{n-3}{2}} dt \\
    & \, \longrightarrow \frac{2^n}{\pi^{\frac{3}{2}}}\, \int_{0}^{\frac{\pi}{2}}  \big(\sin (\theta/2) \big)^{\frac{n}{2}-2}
    \big(\cos (\theta/2) \big)^{\frac{n}{2}-2} d \theta.
\end{align*}  
    Combining all this with Theorem~\ref{projGegenbauerB} and using again the formula \eqref{mainasym} for the asymptotic of ratios of Gamma functions, which implies that
  \[
  \lim_{d \to \infty} \,\, \frac{\Gamma(d+n)}{\Gamma(d+\frac{n+1}{2})d^\frac{n-1}{2}}\,\, = \,\,                     1\,,              \]
gives
    \begin{align*}
    \lim_{d \to \infty}
    \frac{\boldsymbol{\lambda}\big(\mathcal{P}_{d}(\mathbb{S}^{n-1})\big)}{d^\frac{n-2}{2}}
    &
    =\,
    \frac1{2\sqrt{\pi}}\frac{\Gamma(\frac{n}{2})}{\, \Gamma(n-1)}
         \, \left[\lim_{d \to \infty}\,\, \frac{\Gamma(d+n)}{\Gamma(d+\frac{n+1}{2})d^\frac{n-1}{2}}\right]\,\left[ \lim_{d \to \infty} \,2\sqrt{d}\,
    \int_{0}^{1} \big| P_d^{(\frac{n-1}{2},\frac{n-1}{2})}(t)\big| (1-t^2)^{\frac{n-3}{2}} dt\right]
         \\
     &=  \frac{2^{n-1} \Gamma(\frac{n}{2})}{\pi^2 \Gamma(n-1)} \,\, \int_{0}^{\frac{\pi}{2}}  \big(\sin (\theta/2) \big)^{\frac{n}{2}-2}
    \big(\cos (\theta/2) \big)^{\frac{n}{2}-2} d \theta
       \,. 
                            \end{align*}  
  We now check that
  \begin{equation*}
   \text{\bf I}
    := \int_{0}^{\frac{\pi}{2}}  \big(\sin (\theta/2) \big)^{\frac{n}{2}-2}
    \big(\cos (\theta/2) \big)^{\frac{n}{2}-2} d \theta =
    \frac{\sqrt{\pi}}{2^{\frac{n}{2}-1}}
    \frac{\Gamma(\frac{n}{4}-\frac{1}{2})}{\Gamma(\frac{n}{4})}\,.
  \end{equation*}
  Indeed, using  the identity  $ 2\sin(\theta/2) \cos (\theta/2) = \sin \theta$ and the substitution
  $t = \sin^2 \theta$ a~simple calculation leads to
  \[
  \text{\bf I} = \frac{1}{2^{\frac{n}{2}-1}}\,\,  \text{\bf B}\Big(\frac{n}{4}-\frac{1}{2},\frac{1}{2}\Big)
  = \frac{1}{2^{\frac{n}{2}-1}}\,\, \frac{\Gamma(\frac{n}{4}-\frac{1}{2})\Gamma(\frac{1}{2})}{\Gamma(\frac{n}{4})}\,,
  \]
  where $\text{\bf B}$ as usual stands for the Beta-function.

Altogether we have, 
    \begin{align*}
    \lim_{d \to \infty}
    \frac{\boldsymbol{\lambda}\big(\mathcal{P}_{d}(\mathbb{S}^{n-1})\big)}{d^\frac{n-2}{2}}
    &
    =\,
  \frac{2^{n-1} }{\pi^2 \Gamma(n-1)\Gamma(\frac{n}{4}-\frac{1}{2})} \,\frac{\sqrt{\pi}\Gamma(\frac{n}{2})}{2^{\frac{n}{2}-1}\Gamma(\frac{n}{4})}\,\,  \\
  &=\, \frac{2^{n-1} }{\pi^2}  \,\frac{\Gamma(\frac{n}{4}+\frac{1}{2})\Gamma(\frac{n}{4}-\frac{1}{2})}{\Gamma(n-1)}\,\, =\,\, \frac{2^{n} }{\pi^2}  \,\frac{\Gamma(\frac{n}{4}+\frac{1}{2})^2}{(\frac{n}{2}-1)\Gamma(n-1)} \,, 
                            \end{align*}
where in the second to last equality we have used \eqref{Legendre duplication formula}  with $x=\frac{n}{4}$.
       \end{proof}

\smallskip

\begin{remark} \label{differences}
  Summarizing we see that projection constants of spaces of homogeneous
  polynomials in the real and complex case behave substantially different. Indeed, we have
  for fixed $n>2$ and  large $d$ that
  \[
  \boldsymbol{\lambda}\big(\mathcal{P}_{ d}(\mathbb{S}^{n-1})\big)
  \,\sim_{c(n)} d^{\frac{n-2}{2}}
  \quad \text{and} \quad
   \boldsymbol{\lambda}\big(\mathcal{P}_{ d}(\mathbb{S}_{\mathbb{C}}^{n-1})\big)
  \,\sim_{c(n)} \,1,
    \]
  whereas for  fixed $d$ and large $n$
  \[
  \boldsymbol{\lambda}\big(\mathcal{P}_{ d}(\mathbb{S}^{n-1})\big)
  \,\sim_{c(d)} n^{\frac{d}{2}}
  \quad \text{and} \quad
  \boldsymbol{\lambda}\big(\mathcal{P}_{ d}(\mathbb{S}_{\mathbb{C}}^{n-1})\big)
  \,\sim_{c(d)} n^{\frac{d}{2}}\,.
    \]
  Here the first result is a consequence of Corollary~\ref{repetition}, the third one
  was proved in Example 6.2 of \cite{defant2009volume},
  whereas the second and fourth statements  are simple consequences of the
  Ryll-Wojtaszczyk formula~\eqref{fascinating}.
   \end{remark}

\subsection{Finite degree polynomials}

We start by applying Theorem~\ref{maino} to the case $n=2$, which yields an explicit expression for the projection constant $\boldsymbol{\lambda}\big(\mathcal{P}_{\le d}(\mathbb{S}^1)\big)$. Recall again that $\mathcal{P}_{ \leq d}(\mathbb{S}^{1})$ 
stands for the Banach space of all complex-valued degree-$d$ polynomials on the one dimensional real euclidean sphere.

It is worth noting that, as a by-product, we recover the well-known Lozinski-Kharshiladze formula for the projection constant of the Banach space  $\textbf{Trig}_{\leq d}(\mathbb{T})$ of all  complex-valued degree-$d$ polynomials
\[
P(z) = \sum_{|k|\leq d} c_k z^n, \quad\, z \in \mathbb{T}
\]
defined on the  circle group $\mathbb{T}$,
equipped with the supremum norm over $\mathbb{T}$. This formula states that $\boldsymbol{\lambda}\big(\textbf{Trig}_{\leq d}(\mathbb{T})\big)$ equals the Lebesgue constant of the Dirichlet kernel (see \cite[IIIB. Theorem 22]{wojtaszczyk1996banach} and~ \cite{natanson1961constructive}).

\medskip

\begin{proposition}
For all $d \in \mathbb{N}$, we have
\[
\boldsymbol{\lambda}\big(\textbf{Trig}_{\leq d}(\mathbb{T})\big) = \boldsymbol{\lambda}\big(\mathcal{P}_{ \leq d}(\mathbb{S}^{1})\big) =
\frac{1}{2\pi} \int_{0}^{2\pi}\bigg| \frac{\sin\big((d+ \frac{1}{2})t \big)}{\sin\big(\frac{t}{2} \big)}\bigg|\,dt\,.
\]
In particular,
\[
\lim_{d \to \infty}  \frac{\boldsymbol{\lambda}\big(\mathcal{P}_{ \leq d}(\mathbb{S}^{1})\big)}{\log d} \,=\, \frac{4}{\pi^2}\,.
\]
\end{proposition}

\medskip

\begin{proof}
We claim that the linear mapping
$\Phi \colon \textbf{\bf Trig}_{\leq d}(\mathbb T) \to \mathcal{P}_{\leq d}(\mathbb{S}^{1})$,  which for  $P(z) = \sum_{|k|\leq d} c_k z^k$ is given by the formula:
\[
\Phi(P)(x, y):=
\sum_{k=1}^ d c_{-k} (x-iy)^k +c_0+
\sum_{k=1}^ d c_k (x+iy)^k, \quad\, (x, y) \in \mathbb{S}^{1}\,,
\]
is an isometric isomorphism. Being obviously isometric, it is injective.
On the other hand, we deduce from  Proposition~\ref{orthodeco} that
\[
\mathcal{P}_{\leq d}(\mathbb{S}^{1}) = \mathcal H_{\leq d} (\mathbb{S}^{1})= \bigoplus_{k \leq d}\mathcal{H}_k(\mathbb{S}^{1})\,.
\]
Hence by Equation~\eqref{dimfor} we see  that  $\text{dim} \,\mathcal{P}_{\leq d}(\mathbb{S}^{1}) = 2d +1=\text{dim} \textbf{\bf Trig}_{\leq d}(\mathbb{T})$. Combining,
$\Phi$ is as desired an isometric isomorphism,  and  as a consequence
\[
\boldsymbol{\lambda}\big(\textbf{Trig}_{\leq d}(\mathbb{T})\big) = \boldsymbol{\lambda}\big(\mathcal{P}_{ \leq d}(\mathbb{S}^{1})\big)\,.
\]
Now we apply Theorem~\ref{maino} and  Equation~\eqref{dimfor} to get
\begin{equation*}
\mathbf{k}_{\mathcal P_{\leq d} (\mathbb S^1)} (e_1,\eta) =
1+ \sum_{\ell=1}^{d} 2 L_{2,\ell}^{\diamond}(\langle e_1, \eta \rangle), \quad \eta \in
\mathbb S^1\,.
\end{equation*}
  Then similar to the proofs of   Propositions~\ref{neu} and \ref{prop : 2 dimensional homogeneous} we have
  \begin{align*}
    \boldsymbol{\lambda}\big(\mathcal{P}_{ \leq d}(\mathbb{S}^{1})\big)
    &
    =  \frac{1}{2\pi}\bigg(\int_{0}^{\pi} \Big|1+ 2 \sum_{\ell=1}^{d} \cos (\ell t)\Big|\,dt + \int_{\pi}^{2\pi} \Big|1+ 2 \sum_{\ell=1}^{d} \cos \big(\ell (2\pi - t)\big)\Big|\,dt\bigg),
            \end{align*}
and the argument follows as before.
\end{proof}

\smallskip
For $n >2$ we get the following formula - complementing a result of Ragozin \cite[Theorem 4]{ragozin1971uniform}.

\smallskip

\begin{theorem} \label{starto2}
   For  $n >2$ and $d \in \mathbb{N}$
\[
 \boldsymbol{\lambda}\big(\mathcal{P}_{ \leq d}(\mathbb{S}^{n-1})\big) =
 \frac{\Gamma(\frac{n}{2})}{\sqrt{\pi}\, \Gamma(n-1)}\,
  \frac{\Gamma(d+n-1)}{\Gamma(d+ \frac{n-1}{2})}
 \,
  \int_{-1}^{1} \Big| P_d^{(\frac{n-1}{2},\frac{n-3}{2})}(t)\Big| (1-t^2)^{\frac{n-3}{2}} dt\,.
   \]
  \end{theorem}

\begin{proof}
The proof follows the  lines of the proof of Theorem~\ref{projGegenbauerB}. Again we start with Theorem~\ref{maino} which shows that
\[
\mathbf{k}_{\mathcal{P}_{\leq d}(\mathbb{S}^{n-1})}( e_1, \eta )
= \sum_{j=0}^{d} N_{n,j} \,\,L_{n,j}^{\diamond}(\eta_1)
\]
and
\begin{equation*}
 \boldsymbol{\lambda}\big(\mathcal{P}_{\leq d}(\mathbb{S}^{n-1})\big)
 =
\int_{\mathbb{S}^{n-1}}\,
\big| \mathbf{k}_{\mathcal{P}_{\leq d}(\mathbb{S}^{n-1})}(\langle e_1, \eta \rangle)\big| \,d\sigma_n(\eta)\,.
\end{equation*}
Then, as before,
 by Lemma~\ref{lemm: identity Gegenbauer}  and Equation~\eqref{dimfor} we  have
\begin{align} \label{kern-leqd}
     \mathbf{k}_{\mathcal{P}_{\leq d}(\mathbb{S}^{n-1})}( e_1, \eta )
   =
\frac{\Gamma(\frac{n-1}{2})}{\Gamma(n-1)}\,\, \sum_{j=0}^{d}\,(2j +n-2) \, \frac{\Gamma(n+j-2)}{\Gamma(\frac{n-1}{2}+j)}\,
P_{j}^{(\alpha, \alpha)}(\eta_1)\,,
\end{align}
where $\alpha = \frac{n-3}{2}$. Now we first rewrite this expression getting
\begin{align*}
\mathbf{k}_{\mathcal{P}_{\leq d}(\mathbb{S}^{n-1})}( e_1, \eta )
=
\frac{\Gamma(\frac{n-1}{2})}{\Gamma(n-1)}\,
\sum_{j=0}^{d}
\frac{(2j +\alpha +\alpha +1)\Gamma(j +\alpha +\alpha +1)}{\Gamma(j +\alpha +1)} \,
P_{j}^{(\alpha, \alpha)}(\eta_1)
\,,
\end{align*}
But by
\cite[(4.5.3), p. 71]{szeg1939orthogonal}
we have
\begin{equation*}
  \sum_{j=0}^{d}
\frac{(2j +\alpha +\alpha +1)\Gamma(j +\alpha +\alpha +1)}{\Gamma(j +\alpha +1)} \,
P_{j}^{(\alpha, \alpha)}(\eta_1)
\,=\,
\frac{\Gamma(d +\alpha +\alpha +2)}{\Gamma(d +\alpha +1)}
P_{d}^{(\alpha+1, \alpha)}(\eta_1)\,,
\end{equation*}
which shows
\begin{align*}
\mathbf{k}_{\mathcal{P}_{\leq d}(\mathbb{S}^{n-1})}( e_1, \eta )
=
\frac{\Gamma(\frac{n-1}{2})}{\Gamma(n-1)}\,
\frac{\Gamma(d +\alpha +\alpha +2)}{\Gamma(d +\alpha +1)}
P_{d}^{(\alpha+1, \alpha)}(\eta_1)
=
\frac{\Gamma(\frac{n-1}{2})}{\Gamma(n-1)}\,
\frac{\Gamma(d +n-1)}{\Gamma(d +\frac{n-1}{2})}
P_{d}^{(\frac{n-1}{2}, \frac{n-3}{2})}(\eta_1)
\,.
\end{align*}
Consequently, by \eqref{umrechnen} we have
\begin{align*}
&
 \boldsymbol{\lambda}\big(\mathcal{P}_{\leq d}(\mathbb{S}^{n-1})\big)
 =
 \frac{\Gamma(\frac{n-1}{2})}{\Gamma(n-1)}\,
\frac{\Gamma(d +n-1)}{\Gamma(d +\frac{n-1}{2})}\,
\frac{\Gamma(\frac{n}{2})}{\sqrt{\pi}\, \Gamma(\frac{n-1}{2})}
\,
 \int_{-1}^1\, \big| P_{d}^{(\frac{n-1}{2}, \frac{n-3}{2})}(t)\big|  (1-t^2)^{\frac{n-3}{2}} dt\,,
 \end{align*}
 which completes the proof.
\end{proof}

\smallskip

  Note that for each $n$  and  large  $d$
  $$\dim \mathcal{P}_{ \leq d}(\mathbb{S}^{n-1}) \,\sim_{c(n)}\,  d^{n-1}.$$
 Indeed, this easily follows by combining the Equations \eqref{LH9}, \eqref{ludoAA}
  and~\eqref{asymp dimension ratio}.
  Hence, as in the case of Corollary~\ref{repetition} (see also  Remark~\ref{differences}), this again shows that here the Kadets-Snobar theorem leads to a weak upper estimate
 - in contrast to  the result for homogeneous spherical harmonics from
 Theorem~\ref{thm: asymptotic order real homogeneous harmonics}.
 \smallskip

\begin{corollary} \label{Fin}
For each integer $n > 2$
   \[
  \lim_{d \to \infty} \frac{\boldsymbol{\lambda}\big(\mathcal{P}_{ \leq d}(\mathbb{S}^{n-1})\big)}{d^{\frac{n-2}{2}}}
  \,=\,
  \frac{\Gamma(\frac{n}{2}-1)}{2^{\frac{n}{2}-3}\pi \Gamma(\frac{n}{2}-\frac12)^2}\,
        \,.
  \]
 \end{corollary}

 \begin{proof} The proof is similar to that of Corollary~\ref{repetition}. We use that by \cite[§20,  p. 87]{szeg1933Asymptotische} (see also the main result of \cite{rau1929lebesgueschen})
  for $\alpha > -\frac{1}{2}$
 and $\beta > -1$
  \begin{align*}
   \lim_{d \to \infty} \,\, \sqrt{d}\,\,
                \int_{-1}^{1}
                                 (1-x)^\alpha (1+x)^\beta \big|P^{(\alpha+1, \beta)}_d(x)\big| dx
                                 \,=\,
                                 \frac{2^{\alpha + \beta +2}}{\pi^{\frac{3}{2}}}\,
        \frac{\Gamma(\frac{\alpha}{2}+\frac{1}{4})\Gamma(\frac{\beta}{2}+\frac{3}{4}) }
        {\Gamma(\frac{\alpha + \beta}{2}+1)}\,.
 \end{align*}
 In \cite{szeg1933Asymptotische}  this equality is  in fact  proved as an application of \eqref{doro-eq},
  and in the following we apply it to $\alpha = \beta = \frac{n-3}{2}$. Additionally,  using \eqref{mainasym}, we have
     \[
     \lim_{d \to \infty}  \frac{1}{d^{\frac{n-1}{2}}}
     \,
       \frac{\Gamma(d+n-1)}{\Gamma(d+ \frac{n-1}{2})} \, =\, 1\,.
     \]
   Combining with Theorem~\ref{starto2}, we get
      \begin{align*}
    &
    \lim_{d \to \infty}
    \frac{\boldsymbol{\lambda}\big(\mathcal{P}_{ \leq d}(\mathbb{S}^{n-1})\big)}{d^\frac{n-2}{2}}
    \\&
    =
    \frac{\Gamma(\frac{n}{2})}{\sqrt{\pi}\, \Gamma(n-1)}\,
         \,
     \lim_{d \to \infty}  \frac{1}{d^{\frac{n-1}{2}}}
     \,
       \frac{\Gamma(d+n-1)}{\Gamma(d+ \frac{n-1}{2})}
     \,
     \lim_{d \to \infty} \,\, \sqrt{d}\,
         \int_{-1}^{1} \big| P_d^{(\frac{n-1}{2},\frac{n-3}{2})}(t)\big| (1-t^2)^{\frac{n-3}{2}} dt
        \\&
        =
        \frac{\Gamma(\frac{n}{2})}{\sqrt{\pi}\, \Gamma(n-1)}\,
         \,          \,
     \frac{2^{n-1}}{\pi^{\frac{3}{2}}}\,
        \frac{\Gamma(\frac{n}{4}-\frac{1}{2})\Gamma(\frac{n}{4})}{ \Gamma(\frac{n}{2}-\frac{1}{2})}
        =\frac{2^{n-1}\Gamma(\frac{n}{2})}{\pi^2\, \Gamma(n-1)}\,
        \frac{\Gamma(\frac{n}{4}-\frac{1}{2})\Gamma(\frac{n}{4})}{ \Gamma(\frac{n}{2}-\frac{1}{2})}
    \,.                         \end{align*}
        Then applying the Legendre duplication formula
~\eqref{Legendre duplication formula} twice, with  $x = \frac{n}{4} -\frac{1}{2} $  and with $x = \frac{n}{2} -\frac{1}{2} $
concludes the argument.
 \end{proof}

\smallskip

We finish with the following  special case $n=3$ of Theorem~\ref{starto2} and its  Corollary~\ref{Fin}.
Note that according to Theorem~\ref{main-ibk} we here may replace the projection constant of  $\mathcal{P}_{ \leq d}(\mathbb{S}^{2})$ by the norm of the orthogonal projection $\pi_{\mathcal{P}_{ \leq d}(\mathbb{S}^{2})} \colon C(\mathbb{S}^{2}) \to \mathcal{P}_{ \leq d}(\mathbb{S}^{2})$.
In this form the result is due to Gronwall \cite[Equations (7), (8) and (27)]{gronwall1914degree}
(see also  \cite[Section 4.2.]{atkinson2012spherical}).
\smallskip
\begin{corollary}  
For  each $d \in \mathbb{N}$
\[
\boldsymbol{\lambda}\big(\mathcal{P}_{ \leq d}(\mathbb{S}^{2})\big)
= \frac{d+1}{2} \int_{-1}^{1}
|P_d^{(1,0)}(t)| dt\,.
\]
Moreover,
\begin{align*}
\lim_{d \to \infty} \frac{\boldsymbol{\lambda}\big(\mathcal{P}_{ \leq d}(\mathbb{S}^{2})\big)}{\sqrt{d}}
\,=\,
2 \sqrt{\frac{2}{\pi}}\,.
\end{align*}
\end{corollary}

\section{Transition to Real Coefficients} \label{final remark}

Throughout the article, the function spaces under consideration were defined on the real
unit sphere $\mathbb{S}^{n-1}$ but took values in the complex field $\mathbb{C}$. In particular,
all polynomials carried complex coefficients and each space was viewed as a complex
Banach space equipped with the supremum norm on $\mathbb{S}^{n-1}$.

It is natural to ask what becomes of the theory when one restricts attention to
\emph{real}-valued functions, that is, when all coefficients in the defining polynomials are
required to be real. In this setting, the analogous spaces have the same dimensions as
$\mathbb{R}$-vector spaces (i.e. the dimension of the complex vector space of $\mathbb C$-valued polynomials coincides with the dimension of the real vector space of $\mathbb R$-valued polynomials), and all arguments in the paper go through verbatim: the
reproducing kernels involved in the various projections are real-valued polynomials, so all
integral formulas for the corresponding projection constants remain unchanged.

\smallskip

Nevertheless, the Banach-space geometry of the underlying spaces may change. A striking
example is provided by the degree-$d$ harmonics on the circle. As noted in the comments
following Proposition~\ref{neu}, the complex version $\mathcal{H}_{d}(\mathbb{S}^{1})$ is
isometrically isomorphic to the complex space $\ell_1^2(\mathbb{C})$, whereas its
real-coefficient counterpart turns out to be a $2$-dimensional real Hilbert space.  
We give a short proof of this fact.

\medskip

In the real setting, we restrict attention to real-valued harmonics on $\mathbb{S}^1$ and
view $\mathcal{H}_{d}(\mathbb{S}^{1})$ as a vector space over $\mathbb{R}$. The real-valued degree-$d$ harmonics are exactly the functions of the form
\[
f(\cos\theta,\sin\theta)=a\cos(d\theta)+b\sin(d\theta), \qquad a,b\in\mathbb{R}.
\]
Thus
\[
\mathcal{H}_{d}(\mathbb{S}^{1})
=\operatorname{span}_{\mathbb{R}}\{\cos(d\theta),\,\sin(d\theta)\},
\qquad \dim_{\mathbb{R}}\mathcal{{H}}_{d}(\mathbb{S}^{1})=2.
\]

Let $a,b\in\mathbb{R}$ and write $R=\sqrt{a^2+b^2}$, choosing $\varphi\in\mathbb{R}$ so that
$a=R\cos\varphi$ and $b=R\sin\varphi$. Then for every $\theta\in\mathbb{R}$,
\[
a\cos(d\theta)+b\sin(d\theta)=R\cos(d\theta-\varphi),
\]
and therefore
\[
\|f\|_{\infty}
=\sup_{\theta\in\mathbb{R}}|R\cos(d\theta-\varphi)|
=R=\sqrt{a^2+b^2}.
\]
Hence the mapping
\[
\ell_2^2(\mathbb{R})\ni(a,b)\longmapsto a\cos(d\cdot)+b\sin(d\cdot)\in\mathcal{H}_{d}(\mathbb{S}^{1})
\]
is an isometric linear bijection of real Banach spaces. In particular, the real version of
$\mathcal{H}_{d}(\mathbb{S}^{1})$ is a $2$-dimensional Hilbert space.

\smallskip

This distinction in Banach-space structure, however, does not affect the numerical value
of the projection constant. This can be checked directly by comparing the expression in
Proposition~\ref{neu} with the case $n=2$ in~\eqref{beginning}.

As another example illustrating these remarks, recall from~\eqref{aprilo} that
$\mathcal{P}_{d}(\mathbb{S}^{n-1})=\mathcal{P}_{d}(\ell_2^n(\mathbb{R}))$. For $d=1$,
Theorem~\ref{projGegenbauerB} yields a projection constant which agrees with the value
predicted by Rutovitz’s formula for the real Hilbert space appearing in~\eqref{beginning}.

\section*{Acknowledgments}

We thank the referee for providing a more elementary argument for Proposition~\ref{cor:unique_projection}, for drawing our attention to it, and for reading the manuscript with great care and helping us improve its presentation. The second author would like to express gratitude to Professor Mariano Su\'arez-\'Alvarez for bringing reference~\cite{aptekarev1995asymptotic} to our attention. He also thanks Professor Luciano Abadias for clarifying a calculation in his article~\cite{abadias2016quadrature}, which proved instrumental for our purposes. The fourth author would like to thank Professor Dorothee Haroske for her assistance in obtaining and providing the book~\cite{szeg1933Asymptotische}.

\end{document}